\documentclass[article,11pt]{amsart}
\usepackage[top=25mm,bottom=25mm,left=25mm,right=25mm]{geometry}

\newtheorem{lemma}{Lemma}[section]
\newtheorem{proposition}{Proposition}[section]
\newtheorem{theorem}{Theorem}[section]
\newtheorem{corollary}{Corollary}[section]
\newtheorem{definition}{Definition}[section]
\newtheorem{example}{Example}[section]
\newtheorem{remark}{Remark}[section]
\newtheorem{assumption}{Assumption}[section]

\makeatletter
\def\section{\@startsection{section}{1}%
\z@{1\linespacing\@plus\linespacing}{1\linespacing}%
{\bf\centering}}
\def\subsection{\@startsection{subsection}{0}%
\z@{\linespacing\@plus\linespacing}{\linespacing}%
{\bf}}
\makeatother

\makeatletter
\@addtoreset{equation}{section}
\makeatother

\DeclareMathOperator{\supp}{supp}

\DeclareMathOperator{\Spec}{Spec}

\DeclareMathOperator{\dist}{dist}
\DeclareMathOperator{\card}{card}
\DeclareMathOperator{\loc}{loc}

\DeclareMathOperator{\Ai}{{\rm Ai}}

\providecommand{\pro}[1]{(#1_t)_{t \geq 0}}
\providecommand{\proo}[1]{(#1_t)_{t \in \R}}

\newcommand{\cQ}{\mathcal{Q}}
\newcommand{\cG}{\mathcal{G}}

\newcommand{\cA}{\mathcal{A}}
\newcommand{\cL}{\mathcal{L}}
\newcommand{\cK}{\mathcal{K}}
\newcommand{\cF}{\mathcal{F}}
\newcommand{\cB}{\mathcal{B}}

\newcommand{\cT}{\mathcal{T}}

\newcommand{\cM}{\mathcal{M}}

\newcommand{\R}{\mathbf{R}}
\newcommand{\1}{\mathbf{1}}

\newcommand{\pr}{\mathbf{P}}
\newcommand{\qpr}{\mathbf{Q}}

\newcommand{\ex}{\mathbf{E}}
\newcommand{\Rd}{\mathbf{R}^d}
\newcommand{\N}{\mathbf{N}}

\begin{document}
\title[]
{Fractional $P(\phi)_1$-processes and Gibbs measures}
\author{Kamil Kaleta {\tiny{and}} J\'ozsef L{\H o}rinczi}
\address{Kamil Kaleta, Institute of Mathematics and Computer Science \\ Wroc{\l}aw University of Technology
\\ Wyb. Wyspia{\'n}skiego 27, 50-370 Wroc{\l}aw, Poland}
\email{kamil.kaleta@pwr.wroc.pl}

\address{J\'ozsef L\H orinczi,
School of Mathematics, Loughborough University \\
Loughborough LE11 3TU, United Kingdom}
\email{J.Lorinczi@lboro.ac.uk}

\begin{abstract}
We define and prove existence of fractional $P(\phi)_1$-processes as random processes generated by fractional
Schr\"odinger semigroups with Kato-decomposable potentials. Also, we show that the measure of such a process
is a Gibbs measure with respect to the same potential. We give conditions of its uniqueness and characterize
its support relating this with intrinsic ultracontractivity properties of the semigroup and the fall-off of
the ground state. To achieve that we establish and analyze these properties first.

\end{abstract}

\thanks{The first named author was supported by the Polish Ministry of Science and Higher Education grant N N201
527338.}

\maketitle

\baselineskip 0.5 cm

\bigskip\bigskip
\section{Introduction}
The Feynman-Kac formula was originally derived to obtain a representation of the solutions of the
Schr\"odinger equation by running a Brownian motion subject to the given potential and averaging
over the paths. This probabilistic method proved to be a powerful alternative to the direct operator
analysis in studying the properties of the eigenfunctions of Schr\"odinger operators. Feynman-Kac-type
formulae were subsequently extended to cover further PDE and also other models of quantum theory by
adding extra operator terms (see a systematic discussion in \cite{bib:LHB}). Due to the presence of 
the Laplace operator, however, random processes with continuous paths remained a key object in these
functional integral representations.

In the recent paper \cite{bib:HIL} generalized Schr\"odinger operators of the form
\begin{equation}
H = \Psi(-\Delta) + V
\end{equation}
have been introduced, where $\Psi$ is a so called Bernstein function. An example to this class are
the fractional Schr\"odinger operators
\begin{equation}
H_\alpha =(-\Delta)^{\alpha/2} + V, \quad 0 < \alpha < 2.
\end{equation}
These operators are non-local and have markedly different properties from usual Schr\"odinger operators
(obtained for $\Psi(x) = x$). Due to the fact that Bernstein functions with vanishing right limits at
the origin are in a one-to-one correspondence with subordinators, the operators $\Psi(-\Delta)$ generate
subordinate Brownian motion. These are L\'evy processes with c\`adl\`ag paths (i.e., right continuous
paths with left limits) having jump discontinuities. In particular, the fractional Laplacian generates
a symmetric $\alpha$-stable process $\pro X$, and for fractional Schr\"odinger operators a
Feynman-Kac-type formula of the form
\begin{equation}
\label{fk}
\left(e^{-tH_\alpha}f\right)(x) = \ex^x\left[e^{-\int_0^t V(X_s)ds}f(X_t)\right] =: (T_tf)(x), \quad t > 0,
\end{equation}
holds, where the expectation is taken with respect to the measure of this process.

The main goal of this paper is to obtain a description of symmetric $\alpha$-stable processes under the
potential $V$. The Feynman-Kac semigroup $\{T_t: t\geq 0\}$ has the particularity that in general
$T_t\1_{\R^d}(x) \neq 1$. Suppose $V$ is chosen so that there exist $\lambda_0 = \inf \Spec H_\alpha$ and
$\varphi_0 \in L^2(\R^d, dx)$ such that $H_\alpha \varphi_0 = \lambda_0\varphi_0$. Then the intrinsic fractional
Feynman-Kac semigroup generated by the operator $\widetilde H_\alpha f := \frac{1}{\varphi_0} (H_\alpha - \lambda_0)(\varphi_0f)$
is a Markov semigroup and allows a probabilistic interpretation. By treating the exponential factor in (\ref{fk})
as a density with respect to the measure of this semigroup we show that there exists a probability measure $\mu$
and a random process $(\widetilde X_t)_{t\in \R}$ on the space $(D_{\rm r}(\R,\R^d),{\cB}(D_{\rm r}(\R,\R^d))$ of
two-sided c\`adl\`ag paths such that
\begin{equation}
(e^{-t \widetilde{H}_{\alpha}}f)(x) = \ex^x_\mu [f(\widetilde X_t)], \quad t \geq 0.
\end{equation}
We call the Markov process $(\widetilde X_t)_{t \in \R}$ \emph{fractional $P(\phi)_1$-process} for $V$
(Theorem \ref{th:exphi1} below). Note that in order to define this process we need neither positivity nor
boundedness of the potential $V$. We will introduce and use the class of fractional Kato-decomposable
potentials, which allows local singularities, and we will assume that $V$ is such that a ground state
$\varphi_0$ exists. The almost sure behaviour of the measure of this process is established in Theorem
\ref{th:subspaceom}.

Next we show that the stationary measure of a fractional $P(\phi)_1$-process is a Gibbs measure for $V$ on
the paths of this process (Theorem \ref{exist}). We prove that this Gibbs measure is uniquely supported on
the full path space when the fractional Feynman-Kac semigroup is intrinsically ultracontractive (IUC) for
at least large enough times (Theorem \ref{th:iugibbs}). This justifies to introduce the concept of asymptotic
intrinsic ultracontractivity (AIUC), which turns out to be a weaker property than IUC. We characterize AIUC
and IUC for fractional Kato-decomposable potentials (Theorem \ref{th:charIUC}), establish necessary and
sufficient conditions (Theorems \ref{th:nec} and \ref{th:suff}), and show that the borderline case is given,
roughly, by potentials growing faster than logarithmically (Corollary \ref{border}). This contrasts the case
of Schr\"odinger semigroups and diffusions where the classic result \cite{bib:DS} shows that IUC is obtained
for potentials growing at infinity faster than quadratically, and we give a heuristic explanation why is it
``easier" for a fractional $P(\phi)_1$-process to be IUC than for diffusions and what determines the borderline
cases (Remark \ref{whyeasier}). For potentials that are not pinning strongly enough to allow IUC we identify a
full measure subset of c\`adl\`ag paths on which the Gibbs measure is unique (Theorem \ref{th:subspace}). This
subset of paths will be seen to relate with the decay properties of the ground state at infinity. Therefore we
need to derive pointwise lower and upper bounds of the ground states (Theorem \ref{th:eig} and corollaries),
which will also be used to establish (A)IUC for the class of potentials we use.


We note that using these results, one of the applications we are interested in is to add further operators and
study ground state properties of Hamiltonians describing (semi)relativistic quantum field and other models
extending the results of \cite{bib:BHLMS,bib:HIL,bib:HL,bib:LHB}.) This will be discussed elsewhere.

The paper is organized as follows. Section 2 contains essential preparatory material. We introduce two-sided
symmetric $\alpha$-stable processes, recall a minimum of basic definitions and facts on the potential theory
of stable processes and bridges, and derive some results on potential theory for fractional Schr\"odinger
operators with Kato-decomposable potentials. In Section 3 we derive ground state estimates for Kato-decomposable
potentials for which the Feynman-Kac semigroup is compact. Section 4 is devoted to discussing ultracontractivity
properties. In Section 5 we finally prove existence and properties of fractional $P(\phi)_1$-processes. Also, we
construct Gibbs measures on the paths of these processes, and establish their uniqueness and support properties.

\section{Preliminaries}
\subsection{Two-sided symmetric $\alpha$-stable processes}

Let $\pro X$ be an $\Rd$-valued rotationally invariant
$\alpha$-stable process with $d \geq 1$ and $\alpha \in (0,2)$. In this paper we are interested in the
case of non-Gaussian stable processes only, therefore do not include the case $\alpha=2$. We use the
notations $\pr^x$ and $\ex^x$, respectively, for the distribution and the expected value of the process
starting in $x \in \Rd$ at time $t=0$; for simplicity we do not indicate the measure in subscript
(while we do when have any other measure or process). The characteristic function of $\pro X$ is
\begin{equation}
\ex^0[e^{i \xi \cdot X_t}] = e^{-t|\xi|^\alpha}, \quad \xi \in \Rd, \, t \geq 0.
\label{charstable}
\end{equation}
Denote $[0,\infty)=\R^+$. As a L\'evy process, $\pro X$ has a version with paths in $D_{\rm r}(\R^+;\R^d)$,
i.e., the space of right continuous functions $\R^+ \to \R^d$ with left limits (c\`adl\`ag functions)
and in $D_{\rm l}(\R^+; \R^d)$, i.e., the space of left continuous functions $\R^+ \to \R^d$ with right
limits (c\`agl\`ad functions).

The transition density $p(t,x)$ of the process $(X_t)_{t\geq0}$ is a smooth real-valued function on $\R^d$
determined by
$$
\int_{\R^d} p(t,z) e^{i z \xi} dz = e^{-t |\xi|^{\alpha}}, \quad  \xi \in \Rd, \; t > 0,
$$
and $\pr^x(X_t \in A) = \int_A p(t,y-x)dy$ holds for every Borel set $A \subset \R^d$. For every fixed
$t>0$ the density $p(t,x)$ is strictly positive, continuous and bounded on $\R^d$ with the bounds
\begin{align}
\label{eq:weaksc}
C^{-1}\left(\frac{t}{|x|^{d+\alpha}}\wedge t^{-d/\alpha}\right) & \leq p(t,x)
\leq C \left(\frac{t}{|x|^{d+\alpha}}\wedge t^{-d/\alpha} \right).
\end{align}

\noindent
Also, for every $\alpha \in (0,2)$ the scaling property $p(t,x) = t^{-d/\alpha} p(1,t^{-1/\alpha}x)$,
$x \in \R^d$, $t > 0$ holds. 

The L\'evy measure of the process $(X_t,\pr^x)_{t\geq 0}$ is given by
$$
\nu(dx) = \cA_{d,-\alpha} |x|^{-d-\alpha}dx,
$$
where $\cA_{d,\gamma}=2^{-\gamma} \pi^{-d/2} \Gamma((d-\gamma)/2)|\Gamma(\gamma/2)|^{-1}$. For the remainder
of the paper we will simply write $\cA$ instead of $\cA_{d,-\alpha}$.

It is known that when $\alpha < d$, the process $\pro X$ is transient with potential kernel
\cite{bib:BlG}
$$
{\Pi}_\alpha(y-x) = \int_0^\infty p(t,y-x)dt = \cA_{d,\alpha} |y-x|^{\alpha-d}, \quad x,y \in \R^d.
$$
Whenever $\alpha \geq d$ the process is recurrent (pointwise recurrent when $\alpha > d =1$). In this case
we can consider the compensated kernel \cite{bib:BlGR}, that is, for $\alpha \geq d$ we put
$$
{\Pi}_\alpha(y-x) = \int_0^\infty \left( p(t,y-x) - p(t,x_0)\right)dt,
$$
where $x_0 = 0$ for $\alpha > d = 1$, and $x_0 = 1$ for $\alpha = d = 1$. In this case
$$
{\Pi}_\alpha (x) = \frac{1}{\pi} \log\frac{1}{|x|}
$$
for $\alpha = d = 1$ and
$$
{\Pi}_\alpha (x) = (2\Gamma(\alpha) \cos(\pi\alpha/2))^{-1} |x|^{\alpha-1}, \ \ \ \ x \in \R^d
$$
for $\alpha > d = 1$. For further information on the potential theory of stable processes we refer to
\cite{bib:CS1, bib:BBKRSV}.

Below we consider stable processes $\pro X$ extended over the time-line $\R$ instead of defining them
only on the semi-axis $\R^+$ as usual. Consider the measurable space $({\Omega}, {\cB}({\Omega}))$, with
${\Omega}=D_{\rm r}(\R;\R^d)$, as well as $\widehat \Omega = D_{\rm r}(\R^{+}, \R^d) \times
D_{\rm l}(\R^{+}, \R^d)$ and $\widehat \pr^x = \pr^x\times\pr^x$. Let $\omega=(\omega_1,\omega_2)\in
\widehat \Omega$ and define
$$
\widehat X_t(\omega)=
\left\{
\begin{array}{ll}\omega_1(t), & t\geq0,\\
\omega_2(-t),                 & t<0.
\end{array} \right.
$$
Since $\widehat X_t(\omega)$ is c\`adl\`ag in $t \in \R$ under $\widehat \pr^x$,
$X: (\widehat \Omega,{\cB}(\widehat \Omega)) \to (\Omega, {\cB}(\Omega))$ can be defined by $X_t(\omega)=
\widehat X_t(\omega)$. It is seen that $X \in {\cB}(\widehat \Omega)/{\cB}({\Omega})$ by showing that
$X^{-1}(E)\in {\cB}(\widehat \Omega)$, for any cylinder sets $E\in {\cB}({\Omega})$. Thus $X$ is an
$\Omega$-valued random variable on $\widehat \Omega$. Denote again the image measure of $\widehat \pr^x$
on $({\Omega},{\cB}({\Omega}))$ with respect to $X$ by
$$
\pr^x = \widehat \pr^x \circ X^{-1}.
$$
The coordinate process denoted by the same symbol
\begin{equation}
\label{furusato}
X_t:\omega\in {\Omega} \mapsto \omega(t)\in \R^d
\end{equation}
is an \emph{$\alpha$-stable process over $\R$} on $({\Omega}, {\cB}({\Omega}), \pr^x)$, which we denote by
$(X_t, \pr^x)_{t \in \R}$. The properties of the so obtained process can be summarized as follows.
\begin{proposition}
The following hold:
\begin{itemize}
\item[(1)]
$\pr^x(X_0=x)=1$
\item[(2)]
the increments $(X_{t_i}-X_{t_{i-1}})_{1\leq i\leq n}$ are independent symmetric $\alpha$-stable random
variables for any $0=t_0<t_1<\cdots<t_n$ with $X_t-X_s\stackrel{\rm d}{=}X_{t-s}$ for
$t>s$
\item[(3)]
the increments $(X_{-t_{i-1}}-X_{-t_i})_{1\leq i\leq n}$ are independent symmetric $\alpha$-stable random
variables for any $0=-t_0>-t_1>\cdots>-t_n$ with $X_{-t}- X_{-s}\stackrel{\rm d}{=}X_{s-t}$
for $-t>-s$
\item[(4)]
the function $\R \ni t\mapsto X_t(\omega)\in \R$ is c\`adl\`ag for every $\omega$
\item[(5)]
$X_t$ and $X_s$ for $t>0$ and $s<0$ are independent.
\end{itemize}
\end{proposition}

It can be checked directly through the finite dimensional distributions that the joint distribution of
$X_{t_0},\ldots,  X_{t_n}$, $-\infty<t_0<t_1<\ldots<t_n<\infty$ with respect to $dx \otimes d \pr^x$ is
invariant with respect to time shift, i.e.,
$$
\int_{\R^d} dx \ex^{x} \left[\prod_{i=0}^n f_i(X_{t_i})\right] =
\int_{\R^d} dx \ex^{x} \left[\prod_{i=0}^n f_i(X_{t_i+s})\right]
$$
for all $s \in\R$. Moreover, the left hand side above can be expressed in terms of $(X_t, \pr^x)_{t \geq 0}$
as
$$
\int_{\R^d} dx \ex^x \left[\prod_{i=0}^n f_i(X_{t_i})\right] =
\int_{\R^d} dx \ex^x \left[\prod_{i=0}^n f_i(X_{t_i-t_0})\right].
$$

We also will need to consider the process $(X_t)_{t \geq s}$ starting at an arbitrary time $s \in \R$. For
$s,t \in \R$ and $x,y \in \R^d$ we denote its transition density by
\begin{equation*}
p(s,x,t,y) =
\begin{cases}
p(t-s,y-x) & \text{for} \quad s<t \\
0		& \text{for} \quad s \geq t. \\		
\end{cases}
\end{equation*}
By $\pr^{x,s}$ and $\ex^{x,s}$ we respectively denote the distribution and expectation of the process $(X_t)_{t \geq s}$
starting at the point $x \in \R^d$ at time $s \in \R$. We have
$$
\pr^{x,s}(X_t \in A) = \int_A p(s,x,t,y)dy,
$$
where by $\pro X$ we mean the canonical right continuous coordinate process evaluated at time $t >s$, and
$A \in \R^d$ is a Borel set. When $s = 0 $, we simply write $\pr^x$ and $\ex^x$ as before. The following time
translation and scaling properties hold:
$$
(X_t, \pr^{x,s}) \stackrel{\rm d}{=} (X_{t-s},\pr^x), \qquad
(X_t,\pr^{x,s}) \stackrel{\rm d}{=} (rX_{r^{-\alpha}t}, \pr^{xr^{-1},sr^{-\alpha}}), \quad r>0 .
$$

\subsection{Stable bridges}

Let $I \subset \R$ be an interval, and denote by $\Omega_I = D_{\rm r}(I,\R^d)$ the space of c\`adl\`ag
functions from $I$ to $\R^d$. We denote by $\cF_I$ the $\sigma$-field generated by the coordinate process
$\omega(t)$, $\omega \in \Omega_I$, $t \in I$.

For $x,y \in \R^d$ and $s, t \in \R$, $s < t$, we respectively denote by $\pr^{x,s}_{y,t}$ and
$\ex^{x,s}_{y,t}$ the distribution and expectation of the symmetric $\alpha$-stable bridge
$(X_r)_{s \leq r \leq t}$ starting in $x \in \R^d$ at time $s \in \R$ given by $X_t = y$ (see
\cite[Th.1, Th.5]{bib:ChB}, also \cite{bib:FPY}, \cite[Sect. VIII.3]{bib:Ber}). In
fact, $(\pr^{x,s}_{y,t})_{y \in \R^d}$ is a regular version of the family of conditional probability
distributions $\pr^{x,s}(\:\cdot\:| X_t = y)$, $y \in \R^d$, that is, if $Y \geq 0$ is
$\cF_{[s,t]}$-measurable and $g\geq0$ is a Borel function on $\Rd$, then \cite[(2.8)] {bib:FPY}
\begin{align}
\label{eq:cnd}
\ex^{x,s}[Y g(X_t)] = \int_{\Rd} \ex^{x,s}_{y,t}[Y]g(y)p(t-s,y-x)dy.
\end{align}
Clearly, $\pr^{x,s}_{y,t}(X_s = x, X_t = y)=1$.

For $x,y \in \R^d$ and $s, t \in \R$, $s < t$, we denote by $\nu^{x,y}_{[s,t]}$ the non-normalized measure
on $(\Omega_{\left[s,t\right]}, \cF_{\left[s,t\right]})$ corresponding to the symmetric $\alpha$-stable
bridge $(X_r)_{s \leq r \leq t}$ given by
\begin{equation}
\begin{split}
\label{eq:xmeas}
\nu^{x,y}_{[s,t]}(\:\cdot\:) = p(t-s,y-x) \pr^{x,s}_{y,t}(\:\cdot\:).
\end{split}
\end{equation}
Thus for $s = t_0 < t_1 < t_2 < ... < t_n < t_{n+1} = t$ and Borel sets $A_1, A_2, ..., A_n \subset
\R^d$ we have
\begin{equation}
\begin{split}
\label{eq:xmeas1}
\nu^{x,y}_{[s,t]}  (\omega(t_1) \in A_1, & \;\omega(t_2) \in A_2, ..., \omega(t_n) \in A_n) \\
& = \int_{A_1} ... \int_{A_n} \prod_{i=1}^{n+1} p(t_i - t_{i-1}, z_i - z_{i-1}) dz_1...dz_n,
\end{split}
\end{equation}
where $z_0=x$ and $z_{n+1}=y$. Since $\nu^{x,y}_{[s,t]}$ is a measure defined on the set of right
continuous paths with left limits, we may also identify $\nu^{x,y}_{[s,t]}$ as a measure on
$(\Omega_{\R}, \cF_{[s,t]})$.

\smallskip

\subsection{Fractional Schr\"odinger operator and its Feynman-Kac semigroup}
\noindent
Recall that the operator with domain $H^\alpha(\Rd)= \{ f \in L^2(\R^d): \, |k|^{\alpha}\hat f\in
L^2(\R^d)\}$, $0 < \alpha < 2$, defined by its Fourier transform
$$
\widehat{(-\Delta)^{\alpha/2} f}(k) = |k|^\alpha \hat f(k),
$$
is the \emph{fractional Laplacian} of order $\alpha/2$. It is essentially self-adjoint with core 
$C_0^\infty(\R^d)$, and its spectrum is $\Spec ((-\Delta)^{\alpha/2}) =
\Spec_{\rm\tiny{ess}}((-\Delta)^{\alpha/2})=[0,\infty)$.

Let $V: \R^d \to \R$ be a Borel measurable function. We call $V$ \emph{potential} and view it as a
multiplication operator to define fractional Schr\"odinger operators by choosing it from a suitable
function space. We define the space of potentials we will consider.

\begin{definition}{\bf(Fractional Kato-class)}
\rm{
We say that the Borel function $V: \R^d \to \R$ belongs to the \emph{fractional Kato-class} $\cK^\alpha$
if $V$ satisfies either of the two equivalent conditions
$$
\lim_{\varepsilon \rightarrow 0} \sup_{x \in \R^d} \int_{|y-x|<\varepsilon} |V(y) \Pi_\alpha (y-x)| dy = 0
$$
and
$$
\lim_{t \rightarrow 0} \sup_{x \in \R^d} \int_0^t (P_s |V|)(x) ds = 0.
$$
We write $V \in
\cK_{\loc}^\alpha$ if $V \1_B \in \cK^\alpha$ for every ball $B \subset \R^d$. Moreover, we say that $V$
is a \emph{fractional Kato-decomposable potential} whenever
$$
V = V_{+} - V_{-} \quad  \mbox{with} \quad V_{-} \in \cK^\alpha, \;\; V_{+} \in \cK_{\loc}^{\alpha},
$$
where $V_{+}$ and $V_{-}$ denote the positive and negative parts of $V$, respectively.
}
\label{kato}
\end{definition}
\noindent
For the equivalence of the above conditions see (2.5) in \cite{bib:BB2}. To keep the
terminology simple, in what follows we omit the explicit qualifier ``fractional".

\begin{example}
\rm{
Some examples and counterexamples of Kato-potentials are as follows.
\begin{itemize}
\item[(1)]
\emph{Locally bounded potentials:} Let $V \in L^{\infty}_{\loc}(\R^d)$. Then for all $\alpha \in (0,2)$
we have $V \in \cK_{\loc}^\alpha$ and $V$ is Kato-decomposable.
\item[(2)]
\emph{Locally integrable potentials:} Let $\alpha \in (0,2)$. Then $\cK_{\loc}^\alpha \subset L^1_{\loc}(\R^d)$.
\item[(3)]
\emph{Potentials with local singularities:} Let $k \in \N$, $x_i \in \R^d$, $\beta_i > 0$ and
$\varepsilon_i \in \left\{-1,1\right\}$ for $1 \leq i \leq k$. Then the potential
$$
V(x) = \sum_{i=1}^k \varepsilon_i |x-x_i|^{-\beta_i}
$$
belongs to $\cK^{\alpha}$ whenever each $\beta_i < \alpha$ for $\alpha < d$, and $\beta_i < 1$ for 
$\alpha \geq d = 1$.
\item[(4)]
\emph{Coulomb potential:} Let $d = 3$. In the light of (3) above the Coulomb potential $V(x) =
- \frac{C}{|x|}$ belongs to Kato-class $\cK^{\alpha}$ for $\alpha \in (1,2)$ only.
\end{itemize}
}
\label{ex:katoclass}
\end{example}

\smallskip

\begin{definition}[\textbf{Fractional Schr\"odinger operator for bounded potential}]
\rm{
If $V \in L^{\infty}(\Rd)$ we call
\begin{align}
\label{def:frSch}
H_\alpha := (-\Delta)^{\alpha/2} & + V, \quad 0 < \alpha < 2 \,
\end{align}
\emph{fractional Schr\"odinger operator} with potential $V$. We call the one-parameter operator semigroup
$\{e^{-tH_\alpha}: t\geq 0\}$ \emph{fractional Schr\"odinger semigroup}.
}
\end{definition}

\noindent
The above operator is self-adjoint with core $C_0^\infty(\R^d)$.

We define the \emph{Feynman-Kac functional} for the symmetric $\alpha$-stable process by
$$
e_V(t) := e_V(t)(\omega) = e^{-\int_0^t V(X_s(\omega)) ds}, \quad t>0.
$$
If $V \in \cK^\alpha$, then there are constants  $C^{(0)}_V$, $C^{(1)}_V$ such that
\begin{align}
\label{eq:integrab}
\sup_{x \in \R^d} \ex^x[e_{-|V|}(t)] \leq e^{C^{(0)}_V + C^{(1)}_V t}.
\end{align}
When $V$ is Kato-decomposable, then clearly $e_V(t) \leq e_{-V_{-}}(t)$, and therefore
\begin{align}
\label{eq:FKFest}
\sup_{x \in \R^d} \ex^x[e_V(t)] \leq e^{C^{(0)}_{V_{-}} + C^{(1)}_{V_{-}} t}.
\end{align}
\noindent
Clearly, $V_+$ has a killing effect and $V_-$ has a mass generating effect in the Feynman-Kac functional.

The following theorem states that a Feynman-Kac-type formula for fractional Schr\"odinger operators with
Kato-decomposable potentials holds.
\begin{theorem}[\textbf{Functional integral representation}]
Let $V \in L^\infty(\R^d)$, and $f,g \in L^2(\Rd)$. We have
\begin{align}
\label{eq:FKFfrSch}
(f, e^{-tH_\alpha}g)_{L^2}= \int_{\R^d}
\ex^x \left[\overline{f(X_0)}g(X_t) e^{-\int_0^t V(X_s)ds}\right]dx.
\end{align}
Furthermore, let $V$ be a Kato-decomposable potential and define
$$
(T_t f)(x) := \ex^x\left[e_V(t) f(X_t)\right], \quad t\geq 0.
$$
Then $\{T_t: t\geq 0\}$ is a strongly continuous symmetric semigroup. In particular, there exists a
self-adjoint operator $H$ bounded from below such that $e^{-tH} = T_t$.
\label{FK}
\end{theorem}
\noindent
For a proof we refer to \cite{bib:HIL}. For sufficiently regular potentials $V$ we can define $H_\alpha$ as
an operator sum, while for general Kato-decomposable potentials we use $H$ in the theorem above to define
$H_\alpha$ as a self-adjoint operator.
\begin{definition}[\textbf{Fractional Schr\"odinger operator for Kato-class}]
\rm{
Let $V$ be a Kato decomposable potential. We call $H$ given by Theorem \ref{FK} a \emph{fractional
Schr\"odinger operator for Kato-decomposable potential} $V$. We refer to the one-parameter operator
semigroups $\{e^{-tH_\alpha}: t\geq 0\}$ and $\{T_t:  t\geq 0\}$ as the \emph{fractional Schr\"odinger
semigroup} and \emph{Feynman-Kac semigroup} with Kato-decomposable potential $V$, respectively.
}
\label{fSchKato}
\end{definition}

Kato-decomposable potentials allow good regularity properties of the corresponding Feynman-Kac semigroup.
By \cite[Th. 4.13]{bib:HIL} each $T_t$ is a bounded operator from $L^p(\R^d)$ to $L^q(\R^d)$, for all $1
\leq p \leq q \leq \infty$. Moreover, it can be verified directly that all operators $T_t$ are positivity preserving.
Now we state the existence and basic properties of the kernel for the semigroup $\{T_t: t \geq 0 \}$.	

\begin{lemma}
\label{lm:kernel}
Let $V$ be a Kato-decomposable potential. The following properties hold:
\begin{itemize}	
\item[(1)]
for every fixed $t > 0$ the operator $T_t$ has a bounded integral kernel $u(t,x,y)$, i.e. $T_tf(x) = \int_{\R^d}
u(t,x,y) f(y) dy$, $t>0$, $x \in \R^d$, $f \in L^p(\R^d)$, $1\leq p \leq \infty$;
\item[(2)]
$u(t,x,y) = u(t,y,x)$, for every $t>0$, $x,y \in \R^d$;
\item[(3)]
for every $t >0$, $u(t,x,y)$ is continuous on $\R^d \times \R^d$;
\item[(4)]
$u(t,x,y)$ is strictly positive on $(0,\infty) \times \R^d \times \R^d$;
\item[(5)]
for all $x,y \in \R^d$ and $s,t \in \R$, $s<t$, the functional representation
\begin{align}
\label{eq:FKF}
u(t-s,x,y) = \int e^{-\int_s^t V(X_r(\omega))dr} d\nu^{x,y}_{[s,t]}(\omega),
\end{align}
holds, where the $\alpha$-stable bridge measure $\nu^{x,y}_{[s,t]}$ is given by (\ref{eq:xmeas}).
\end{itemize}
\end{lemma}
\noindent
The proof of this lemma follows by standard arguments based on \cite[Section 3.2]{bib:CZ} and we omit it.

\subsection{Potential theory of fractional Schr\"odinger operators}
\noindent
Here we introduce some potential theoretic tools for fractional Schr\"odinger operators needed for
our purposes, and show some technical lemmas to be used in proving our results concerning intrinsic
ultracontractivity and ground state estimates below. For background we refer to
\cite{bib:BB1, bib:BB2, bib:CS, bib:CS1, bib:CZ}.

The potential operator for the semigroup $\{T_t: t \geq 0\} $ is defined by
\begin{align*}
G^V f(x) = \int_0^\infty T_t f(x) dt = \ex^x \left[\int_0^\infty e_V(t) f(X_t) dt \right],
\end{align*}
for non-negative Borel functions $f$ on $\R^d$. If $\int_0^\infty \left\|T_t\right\|_\infty dt < \infty$,
then by the $L^p$-to-$L^q$ boundedness of $T_t$ it follows that $G^V$ is a bounded operator on $L^p(\R^d)$,
$1 \leq p \leq \infty$. In particular, $G^V \1 \in L^\infty$ and $G^V$ has a symmetric kernel given by
$G^V(x,y)=\int_0^\infty u(t,x,y) dt$, i.e., $G^V f(x)= \int_{\R^d} G^V(x,y)f(y)dy$.

The $V$-Green operator for an open set $D$ is defined by
\begin{align*}
G^V_D f(x)  = \int_0^\infty \ex^x \left[t<\tau_D; e_V(t) f(X_t) \right] dt =
\ex^x \left[\int_0^{\tau_D} e_V(t) f(X_t) dt \right],
\end{align*}
for non-negative Borel functions $f$ on $D$, where $\tau_D = \inf\{t > 0: \, X_t \notin D\}$ is the first 
exit time of the process $(X_t)_{t \geq 0}$ from the set $D$. Denote
$$
v_D(x) = G^V_D \1 (x).
$$

The following technical lemma will be used below.

\begin{lemma}
\label{lm:Greenop}
Let $D \subset \R^d$ be a non-empty bounded open set, and $V$ be a strictly positive and bounded
potential on $D$. Then for all $x \in D$ we have
$$
\left(1-\exp(-\sup_{y \in D} V(y))\right) \frac{\pr^x(\tau_D > 1)}{\sup_{y \in D} V(y)} \leq v_D(x) \leq
\frac{1}{\inf_{y \in D} V(y)}.
$$
\end{lemma}

\begin{proof}
Fix $D \subset \R^d$. To simplify the notation denote $\beta = \sup_{y \in D} V(y)$ and  $\zeta =
\inf_{y \in D} V(y)$. For $x \in D$ we have
\begin{align*}
v_D(x)
& = \ex^x \left[ \int_0^{\tau_D} e^{-\int_0^t V(X_s) ds} dt\right]
\geq \ex^x \left[ \int_0^{\tau_D} e^{- \beta t } dt\right] \\
& = \frac{\ex^x \left[ 1 - e^{-\beta \tau_D  } \right]}
{\beta} \geq \left(1-e^{-\beta}\right) \frac{\pr^x(\tau_D > 1)}{\beta} \, .
\end{align*}
Moreover,
$$
v_D(x) =
\ex^x \left[ \int_0^{\tau_D} e^{-\int_0^t V(X_s) ds} dt\right] \leq
\ex^x \left[ \int_0^{\tau_D} e^{-\zeta t} dt\right] =
\ex^x \left[1-e^{-\zeta \tau_D}\right]\zeta^{-1} \leq \zeta^{-1}.
$$
\end{proof}

Furthermore, if $D'$ is an open set such that $D \subset D' \subseteq \R^d$ and $f$ is a non-negative
Borel function on $D'$, then by the strong Markov property of stable processes we have for every $x \in D$
\begin{equation}
\label{eq:pot1}
\begin{split}
G^V_{D'}f(x)
& =
\ex^x \left[\int_0^{\tau_D} e_V(t) f(X_t) dt \right] +
\ex^x \left[\int_{\tau_D}^{\tau_{D^{'}}} e_V(t) f(X_t) dt \right] \\
& =
G^V_D f(x) + \ex^x \left[e^{-\int_0^{\tau_D} V(X_s) ds} \int_{\tau_D}^{\tau_{D'}}
e^{-\int_{\tau_D}^t V(X_s) ds} f(X_t) dt \right] \\
& =
G^V_D f(x) + \ex^x \left[e_V(\tau_D) \ex^{X_{\tau_D}}\left[ \int_0^{\tau_{D^{'}}}
e_V(t) f(X_t) dt \right]\right] \\
& =
G^V_D f(x) + \ex^x\left[e_V(\tau_D)G^V_{D^{'}}f(X_{\tau_D})\right].
\end{split}
\end{equation}

Define $\Phi(t)=\sup_{x \in \R^d} \ex^x[t<\tau_D;e_V(t)]$, $t>0$. If $\Phi \in L^1(0,\infty)$, then by
standard arguments $G^V_D \1 \in L^\infty$ and $G^V_D$ is given by a symmetric kernel $G^V_D(x,y)$, 
$x,y \in D$, i.e., $G^V_D f(x) = \int_{D} G^V_D(x,y)f(y)dy$ (see \cite[cor.Th.3.18]{bib:CZ} and
\cite[p.58]{bib:BB1}). It can be easily checked that this condition is satisfied when, for instance, 
$V \in \cK^\alpha_{\loc}$, $V \geq C_V > 0$ on $D$. The function $G^V_D(x,y)$ is the $V$-Green function 
of the set $D$.

It is easy to see that if $V \geq 0$ on $D$, then the function $u_D(x) := \ex^x[e_V(\tau_D)]$ is 
bounded in $D$. If $D$ is a bounded domain with the exterior cone property and $u_D(x)$ is bounded in $D$, 
then for $f \geq 0$ we have
\begin{align}
\label{eq:IWF}
\ex^x[e_V(\tau_D)f(X_{\tau_D})] & = \cA \int_D G^V_D(x,y) \int_{D^c} \frac{f(z)}{|z-y|^{d+\alpha}} dz dy,
\quad x \in D,
\end{align}
see \cite[eq. (17), Th. 4.10]{bib:BB1}.

The following estimate will be important below. For any $\gamma \geq 0 $, $\gamma \neq d$, there exists $C_{\gamma}>0$
such that
\begin{align}
\label{eq:cruest}
\int_{B(x,|x|/4)^c}(1+|y|)^{-\gamma}|x-y|^{-d-\alpha}dy & \le C_{\gamma}|x|^{-\gamma'}
\end{align}
for $|x| \geq 1$, where $\gamma'=\min(\gamma + \alpha, d + \alpha)$. The result follows from \cite[Lemma 4]{bib:Kw}
for $\gamma>0$, while for $\gamma = 0$ it is trivial.

The next lemma is a generalization to Kato class of \cite[Lemma 6]{bib:KaKu}, where the result was obtained for
$V \in L^{\infty}_{\loc}$. It concerns the comparability of functions $u_D$ and $v_D$ when $D$
is a ball, and plays a crucial role in the proofs of the main theorems in this section.

\begin{lemma}
\label{lm:est}
Let $V \in \cK^\alpha_{\loc}$, $D = B(x,r)$, $r > 0$ and $0 < \kappa <1$. There exists a constant
$C_{r,\kappa}>0$ such that if $V \geq 0$ on $D$, then
\begin{align}
\label{eq:est}
C_{r,\kappa}^{-1} \, v_D(y) & \le u_D(y) \le C_{r,\kappa} \, v_D(y)
\end{align}
for all $y \in B(x,\kappa r)$, $x \in \R^d$.
\end{lemma}
\begin{proof}
The proof can be done by similar arguments as for its version in case of $V \in L^{\infty}_{\loc}$. However, 
the equality
\begin{align}
\label{eq:riemann}
\int_0^{\tau_D} e_V(t)V(X_t)dt & = 1 - e_V(\tau_D), \quad \pr^z-\mbox{a.s.}, \; z \in \Rd
\end{align}
valid in that case needs to be modified here. To obtain it for $V \in \cK^\alpha_{\loc}$ it suffices to observe 
that $V_D \in \cK^\alpha$ for $V_D = V \1_D$. Then for all $z \in \Rd$ the function $\Phi(t)=V_D(X_t)$ is
$\pr^z$-a.s. locally integrable in $(0,\infty)$ and $e_{V_D}(t)$ is $\pr^z$-a.s. locally absolutely
continuous in $(0,\infty)$, and thus \eqref{eq:riemann} follows.
\end{proof}

By using the above lemma it is possible to extend \cite[Theorem 6]{bib:KaKu} to potentials $V \in
\cK^\alpha_{\loc}$. This implies
the following estimate which will be a crucial step in the proof of the characterization of 
ultracontractivity properties of the fractional Schr\"odinger semigroup below.

\begin{lemma}
\label{cor:bhisc}
Let $V \in \cK^\alpha_{\loc}$. Suppose that there is $R > 0$ such that $V(x) \geq 1$ for $|x| \geq R$.
Then there exists a constant $C>0$ such that if $r > 0$, $x_0 \in \Rd$, $|x_0| - r \ge R$ and $f(x) =
\ex^x[e_V(\tau_{B(x_0,r)})f(X_{\tau_{B(x_0,r)}})]$ for $x \in B(x_0,r)$, $f \ge 0$, then
\begin{align}
\label{eq:bhisc}
f(x) & \le C \int_{B(x_0,r/2)^c}\frac{f(y)}{|y-x_0|^{d+\alpha}}dy
\end{align}
for $x \in B\left(x_0,r/2\right)$.
\end{lemma}
\noindent
Note that the function satisfying the mean-value property as in the lemma above is known as regular 
$V$-harmonic in $B(x_0,r)$ (for more details see \cite[p. 83]{bib:BB1}).

\section{Ground state estimates for fractional Schr\"odinger operators}



\subsection{Ground state}
\noindent
The following is a standing assumption for the remainder of this paper.
\begin{assumption}
\rm{
We assume that $\lambda_0 := \inf \Spec H_{\alpha}$ is an isolated eigenvalue and the corresponding
eigenfunction $\varphi_0$ such that $\|\varphi_0\|_2=1$, called \emph{ground state}, exists.
\label{gs}
}
\end{assumption}

\begin{remark}
\rm{
\hspace{100cm}
\begin{trivlist}
\item[\quad (1)]
\emph{Existence:}
There are few results in the literature on the existence of ground states for fractional Schr\"odinger
operators. In \cite[Th. V.1]{bib:CMS} the case of ``shallow" potentials has been discussed. Specifically,
it is shown that whenever $V$ is non-positive, not identically zero and bounded with compact support,
then $H_{\alpha}$ has a ground state $\varphi_0$ corresponding to the negative eigenvalue $\lambda_0$
if and only if $\pro{X}$ is recurrent, i.e., if $d=1$ and $\alpha \in [1,2)$.
\item[\quad (2)]
\emph{Uniqueness:}
Recall that the non-negative integer ${\mathfrak m}(\lambda_0) = \dim \ker (H_\alpha-\lambda_0)$ is
the multiplicity of the ground state, and whenever ${\mathfrak m}(\lambda_0)=1$, the ground state is said
to be unique.
If $V$ is a Kato-decomposable potential, then $T_t f(x)=\int_{\Rd} u(t,x,y) f(y) dy > 0$ for every positive
$f\in L^2(\R^d)$ by Lemma \ref{lm:kernel} (4), thus the operator $T_t$ is positivity improving, $\forall
t>0$. Then the Perron-Frobenius theorem \cite{bib:RS} implies that ${\mathfrak m}(\lambda_0)=1$ and
$\varphi_0$ has a strictly positive version whenever it exists. 
\end{trivlist}
}
\end{remark}

By similar arguments as in the proof of Lemma \ref{lm:kernel} (3) we can show that $T_t(L^\infty(\Rd)) \subset
C_b(\Rd)$. Since $T_t \varphi_0(x) = \int_{\R^d} u(t,x,y) \varphi_0(x) dx = e^{-\lambda_0t} \varphi_0(x)$
and the operator $T_t:L^2(\Rd) \rightarrow L^\infty(\Rd)$ is bounded, $\varphi_0$ is a continuous and
bounded function. We denote the spectral gap of the operator $H_{\alpha}$ by $\Lambda :=
\inf (\Spec H_{\alpha} \setminus \left\{\lambda_0\right\}) - \lambda_0$. We quote the following well-known lemma
as it will be used below (for a proof see  \cite{bib:BL}).
\begin{lemma}
\label{lm:proj}
For all $t>2$
$$
\sup_{x,y \in \Rd} |u(t,x,y) - e^{-\lambda_0 t}\varphi_0(x)\varphi_0(y)| \leq
C_V e^{-(\Lambda+\lambda_0) t}.
$$
\end{lemma}

\subsection{Compactness of $T_t$}
\noindent
When for every $t > 0$ the operators $T_t$ are compact, the spectrum of $T_t$ is discrete. The
corresponding eigenfunctions $\varphi_n$ satisfy $T_t \varphi_n = e^{-\lambda_n t} \varphi_n$,
where $\lambda_0 < \lambda_1 \le \lambda_2 \le \ldots \to \infty$.
All $\varphi_n$ are bounded continuous functions, and each $\lambda_n$ has finite multiplicity.
Whenever $V$ is non-negative, $\lambda_0 > 0$, however, if $V$ has no definite sign, then it may
happen that $\lambda_0 \leq 0$. In what follows this more general case will be considered.


\begin{lemma}
\label{lm:compact}
Let $V$ be a Kato-decomposable potential. If $V(x) \to \infty$ as $|x| \to \infty$, then for all $t > 0$
the operators $T_t$ are compact.
\end{lemma}
\begin{proof}
For any $x \in \R^d$ denote $D:=B(x,1)$. Let $t > 0$ be fixed. We have
\begin{align*}
T_t\1(x)
& =
\ex^x[e_V(t)] = \ex^x\left[\tau_D \geq t;e^{-\int_0^t V(X_s) ds} \right]
+ \ex^x\left[\tau_D < t;e^{-\int_0^t \left(V_{+}(X_s)-V_{-}(X_s)\right) ds}\right] \\
& \leq
e^{-t \inf_{y \in D} V(y)} + \ex^x\left[e^{-\int_0^{\tau_D} V_{+}(X_s) ds}
e^{\int_0^t V_{-}(X_s) ds} \right] \\
& \leq
e^{-t \inf_{y \in D} V(y)} + \left(\ex^x\left[e^{-\int_0^{\tau_D} 2V_{+}(X_s) ds}\right]\right)^{1/2}
\left(\ex^x\left[e^{\int_0^t 2V_{-}(X_s) ds} \right]\right)^{1/2} \\
& \leq
e^{-t\inf_{y\in D} V(y)}+C_{V,t}\left(\ex^0\left[e^{-2\inf_{y\in D}V(y)\tau_{B(0,1)}}\right]\right)^{1/2}
\end{align*}
by Schwarz inequality. Since $V(x) \rightarrow \infty$ as $|x| \rightarrow \infty$,
$\lim_{|x| \rightarrow \infty}T_t\1(x)=0$ follows.

Let now $(V_{r,t})$, $r >0$, be the family of operators given by the kernels $v_r(t,x,y)=
u(t,x,y)\1_{B(0,r)}(y)$, i.e., $V_{r,t}f(x) = \int_{\R^d}v_r(t,x,y)f(y)dy$, $f \in L^2(\R^d)$. We have
\begin{align*}
\int_{\R^d} \int_{\R^d} (v_r(t,x,y))^2 dxdy
& = \int_{B(0,r)} \int_{\R^d} (u(t,x,y))^2 dxdy \\
& \leq C_{V,t} \int_{B(0,r)} T_t\1(y)dy \leq C_{V,t} e^{C^{(0)}_V + C^{(1)}_V t}  |B(0,r)| < \infty.
\end{align*}
Hence $V_{r,t}$ is a Hilbert-Schmidt operator, thus compact. Furthermore, by Schwarz inequality
\begin{align*}
\left\|T_tf - V_{r,t}f \right\|_2^2
&=
\int_{\R^d} \left|\int_{B(0,r)^c}u(t,x,y)f(y)dy\right|^2 dx \\
&\leq
\int_{\R^d} \int_{B(0,r)^c}u(t,x,y)dy \int_{B(0,r)^c}u(t,x,y)\left|f(y)\right|^2dy dx \\
&\leq
e^{C^{(0)}_V + C^{(1)}_V t} \int_{\R^d} \int_{B(0,r)^c}u(t,x,y)\left|f(y)\right|^2dy dx \\
&=
e^{C^{(0)}_V + C^{(1)}_V t} \int_{B(0,r)^c}\int_{\R^d} u(t,x,y) dx\left|f(y)\right|^2 dy \\
&\leq
C_{V,t} \left\|f\right\|_2^2 \sup_{y \in B(0,r)^c} T_t\1(y).
\end{align*}
Since $\lim_{|x| \rightarrow \infty}T_t\1(x)=0$, it follows that $T_t$ can be approximated by compact
operators $V_{r,t}$ in operator norm. Thus $T_t$ is compact.
\end{proof}

\subsection{Decay of the ground state}
\noindent
Notice that the condition $V(x) \rightarrow \infty$ as $|x| \rightarrow \infty$ implies that
$\supp(V_{-})$ is a bounded set and $V = V_{+} \geq 0$ on $(\supp(V_{-}))^c$. Thus we are able to
make use of the results of Section 3.1 for $V$ and $D = B(x,r)$ such that $D \cap \supp(V_{-}) =
\emptyset$.

\begin{lemma}
\label{lm:cruc}
Let $V$ be a Kato-decomposable potential such that $V(x) \rightarrow \infty$ as $|x| \rightarrow \infty$.
Put $D=B(x,1)$. Let $f$ be a non-negative bounded function on $\R^d$ with the property
\begin{align*}
f(x) & \le  C^{(1)}_V v_D(x)\left(\sup_{y \in B\left(x,|x|/2\right)}f(y)
+ \int_{B\left(x,|x|/2\right)^c} f(z)|z-x|^{-d-\alpha}dz \right) \,
\end{align*}
for any $|x| \geq 3$ such that  $D \cap \supp(V_{-}) = \emptyset$. Then
$$
f(x) \le C^{(2)}_V v_D(x)|x|^{-d-\alpha}
$$
for all $|x| \geq 3$ such that $D \cap \supp(V_{-}) = \emptyset$.
\end{lemma}

\begin{proof}
This can be obtained by an adaptation of the proof of \cite[Lemma 5]{bib:KaKu}.

\end{proof}

For $\eta>0$ denote $V_{\eta} = V + \eta$ and
$$
v_{D,\eta}(x) = \ex^x \left[\int_0^{\tau_D} e_{V_{\eta}}(t) dt\right].
$$
This implies that $v_{D,\eta} = G^{V_\eta}_D \1$. The following theorem gives sharp ground state estimates for the Kato-decomposable potential $V =V_{+} - V_{-}$ outside the support of $V_{-}$.

\begin{theorem}[\textbf{Ground state estimates}]
\label{th:eig}
Let $D:=B(x,1)$ and $V$ be a Kato-decomposable potential such that $V(x) \rightarrow \infty$ as $|x|
\rightarrow \infty$. Then for every $\eta \geq 0$ such that $\eta + \lambda_0>0$, there exist constants
$C_{V,\eta}^{(1)}$ and $C_{V,\eta}^{(2)}$ such that if $D \cap \supp(V_{-})=\emptyset$, then
\begin{align}
\label{eq:eig0}
\frac{C_{V,\eta}^{(1)}v_{D,\eta}(x)}{(1+|x|)^{d+\alpha}}
& \le \, \varphi_0(x) \le \frac{C_{V,\eta}^{(2)} v_{D,\eta}(x)}{(1+|x|)^{d+\alpha}}
\end{align}
for every $x \in \R^d$.
\end{theorem}
\begin{proof}
Take $\eta \geq 0$ such that $\lambda_0 + \eta > 0$ (if $\lambda_0>0$, we may take $\eta=0$). Notice that on
integration in the equality
$$
e^{-(\lambda_0 + \eta)t} \varphi_0(x) = e^{-\eta t} T_t \varphi_0(x) = \ex^x[e_{V_{\eta}}(t)\varphi_0(X_t)]
$$
we obtain
$$
\varphi_0(x) = (\lambda_0 + \eta) G^{V_{\eta}}\varphi_0(x).
$$
By \eqref{eq:pot1} applied to $D' = \Rd$ and $f = \varphi_0$ we furthermore get
\begin{align}
\label{eq:refeta}
\varphi_0(x)
& = (\lambda_0 + \eta) G^{V_\eta}_D \varphi_0(x) + \ex^x[e_{V_{\eta}}(\tau_D) \varphi_0(X_{\tau_D})],
\quad x \in D.
\end{align}

First we prove the upper bound. Let $|x| < 3$ be such that $D\cap \supp(V_{-}) = \emptyset$. By
\eqref{eq:refeta} and \eqref{eq:est} we have
\begin{align*}
\varphi_0(x) \leq \left\|\varphi_0\right\|_\infty ((\lambda_0 + \eta) v_{D,\eta}(x) + u_{D,\eta}(x) )
\leq  C_{V,\eta} v_{D,\eta}(x) (1+|x|)^{-d-\alpha}.
\end{align*}
Now let $|x| \ge 3$ be such that $D\cap \supp(V_{-}) = \emptyset$. With $r=\frac{|x|}{2}$, by
\eqref{eq:refeta} and \eqref{eq:IWF} we have
\begin{align*}
\varphi_0(x)
& = (\lambda_0 + \eta) \int_D G^{V_\eta}_D(x,y)\varphi_0(y)dy + \ex^x[X_{\tau_D} \in D^c \cap B(x,r);
e_{V_{\eta}}(\tau_D) \varphi_0(X_{\tau_D})]  \\
& \ \ \ \ + \ex^x[X_{\tau_D} \in B(x,r)^c; e_{V_{\eta}}(\tau_D) \varphi_0(X_{\tau_D})]     \\
& \le
(\lambda_0+\eta)v_{D,\eta}(x)\sup_{y \in B(x,r)}\varphi_0(y)+u_{D,\eta}(x)\sup_{y \in B(x,r)}\varphi_0(y)\\
& \ \ \ \ + \cA \int_D G^{V_\eta}_D(x,y) \int_{B(x,r)^c} \varphi_0(z) |z-y|^{-d-\alpha}dzdy.
\end{align*}
By \eqref{eq:est} furthermore
\begin{align*}
\varphi_0(x)
& \le
(\lambda_0 + \eta) v_{D,\eta}(x) \sup_{y \in B(x,r)}\varphi_0(y)+ C v_{D,\eta}(x) \sup_{y \in B(x,r)}
\varphi_0(y)\\
& \ \ \ \ + C \int_D G^{V_\eta}_D(x,y)dy \int_{B(x,r)^c} \varphi_0(z) |z-x|^{-d-\alpha}dz \\
& \le C_{V,\eta} v_{D,\eta}(x)\left(\sup_{y\in B(x,r)}\varphi_0(y)+\int_{B(x,r)^c}\varphi_0(z)
|z-x|^{-d-\alpha}dz\right)
\end{align*}
follows. On an application of Lemma \ref{lm:cruc} to $f=\varphi_0$ we obtain $\varphi_0(x) \leq C_{V,\eta}
v_{D,\eta}(x)|x|^{-d-\alpha}$ for $|x| \geq 3$ and $D\cap \supp(V_{-}) = \emptyset$. This gives the claimed
upper bound.

To show the lower bound we use \eqref{eq:refeta} again. Let $|x| \leq 2$; then
$$
\varphi_0(x) \geq (\eta + \lambda_0) v_{D,\eta}(x) \inf_{y \in B(0,3)} \varphi_0(y)
\geq C_{V,\eta} v_{D,\eta}(x) (1 + |x|)^{-d-\alpha}.
$$
Take now $|x| > 2$. By (\ref{eq:refeta}) and \eqref{eq:IWF} we have
\begin{align*}
\varphi_0(x)
& \geq
\ex^x[e_{q_{\eta}}(\tau_D) \varphi_0(X_{\tau_D})] =
C \int_D G^{V_\eta}_D(x,y) \int_{D^c} \varphi_0(z) |z-y|^{-d-\alpha}dzdy  \\
& \geq
C \int_D G^{V_\eta}_D(x,y) \int_{B(0,1)} \varphi_0(z) |z-y|^{-d-\alpha}dzdy \geq
C_V v_{D,\eta}(x)|x|^{-d-\alpha}.
\end{align*}
\end{proof}

By using Lemma \ref{lm:Greenop}, we can derive sharp estimates for $v_{D,\eta}(x)$ in many cases of
sufficiently regular potentials. The following corollary gives explicit two-sided bounds on the ground
state for potentials subject to an extra condition.
\begin{corollary}
\label{th:eig1}
Let $V$ be a Kato-decomposable potential such that $V(x) \rightarrow \infty$ as $|x| \rightarrow \infty$.
Moreover, let $A \subset \left\{x \in \R^d: V(y) \geq 1 \ \text{for} \ y \in B(x,1)\right\}$, and $M_{V,A}
\geq 1$ be a constant such that for every $x \in A$ we have
\begin{align}
\label{eq:compcond}
 V(z) & \le M_{V,A} V(y), \quad z, y \in B(x,1).
\end{align}
Then there exist constants $C_{V,A}^{(1)}$ and $C_{V,A}^{(2)}$ such that for
all $x \in A$ the estimates
\begin{align}
\label{eq:eig1}
\frac{C_{V,A}^{(1)}}{V(x)(1+|x|)^{d+\alpha}}
& \le \, \varphi_0(x) \le \frac{C_{V,A}^{(2)}}{V(x)(1+|x|)^{d+\alpha}}
\end{align}
hold.
\end{corollary}

\begin{proof}
First we fix $\eta$ in Theorem \ref{th:eig}. If $\lambda_0 > 0$ put $\eta = 0$, if $\lambda_0 < 0$ put
$\eta = -2\lambda_0$. If $\lambda_0 = 0$, then we choose $\eta=1$. Fix now $x \in A$. Let $D := B(x,1)$
and $M = M_{V,A}$. Observe that by condition \eqref{eq:compcond} we have
$$
M^{-1} \eta \leq M^{-1} ( V(x) + \eta) \leq   \inf_{y \in D} V(y) + \eta \leq  \sup_{y \in D} V(y) +
\eta \leq M (V(x) + \eta).
$$
This and Lemma \ref{lm:Greenop} give
$$
 \frac{M^{'}}{V(x) + \eta} \leq v_{D,\eta}(x) \leq
\frac{M}{V(x) + \eta},
$$
with $M^{'} = M^{-1} \left(1-e^{-M^{-1} \eta}\right)\pr^0(\tau_{B(0,1)} > 1)$, which implies \eqref{eq:eig1}
as a consequence of Theorem \ref{th:eig}.
\end{proof}

\begin{example}
\rm{
We illustrate the above results on some specific cases of $V$.
\begin{itemize}
\item[(1)]
Corollary \ref{th:eig1} can be used to obtain ground state estimates for each of the following potentials:
(i) $V(x) = |x|^{2m}$, $m \in \N$, if $|x| \geq 2$, (ii) $V(x)=|x|^{\beta} \log (1+|x|)$, $\beta >0$, if
$|x| \geq e$, (iii) $V(x)=e^{\beta |x|}$, $\beta >0$, for all $x \in \R^d$, (iv) $V(x)=|x|^{-\beta}e^{|x|}$,
$0< \beta < \alpha < d$ or $0< \beta < 1 = d \leq \alpha$, provided $|x| \geq 1+1/\beta$.


\item[(2)]
Let $V(x)=\1_{\{|x|>1\}}\log|x| - \1_{\{|x| \leq 1\}} \left(|x|^{-\beta} - 1\right)$, for $0<\beta < \alpha < d$ or $0<\beta < 1 = d \leq \alpha$.
Then for $|x| \geq 1+e$
$$
\frac{C_V^{(1)}}{|x|^{d+\alpha}\log|x|} \leq \varphi_1(x) \leq \frac{C_V^{(2)}}{|x|^{d+\alpha}\log|x|}.
$$\item[(3)]
By taking $\alpha = 1$ and $m = 1$ in Example 1(i) we obtain the massless relativistic harmonic oscillator.
In the case $d=1$ the spectral properties of the operator $\sqrt{-d^2/dx^2} + x^2$ are studied in detail in
\cite{bib:LMa}. In particular, the large $x$ asymptotics is established for all the eigenfunctions, and in
particular for the ground state
\begin{eqnarray*}
\varphi_0(x) = \sqrt{\frac{2}{-a_1'}}\left(\frac{p_3(a_1')}{x^4}-\frac{p_5(a_1')}{x^6}+\ldots+
(-1)^{N}\frac{p_{2N-1}(a_1')}{x^{2N}}\right) +O\left(\frac{1}{x^{2(N+1)}}\right)
\end{eqnarray*}
is obtained, where $a_1'\simeq -3.2482$ denotes the first zero of the derivative of the Airy function $\Ai(x)$,
and $p_n, q_n$ are $n$th order polynomials defined by the recursive relations $p_{n+1}(x) = p_n'(x)+xq_n(x)$ and
$q_{n+1}(x) = p_n(x) + q_n'(x)$, with $p_0(x)\equiv 1$, $q_0(x)\equiv 0$. For odd order eigenfunctions the leading
term can be improved to order $x^{-5}$, for even order eigenfunctions it is of order $x^{-4}$ as predicted by
Corollary \ref{th:eig1}.
\end{itemize}
}
\end{example}

Our next result concerns purely negative potentials.

\begin{theorem}
\label{th:decaying}
Let $V$ be a Kato-decomposable potential such that $V_{+} \equiv 0$ and $V_{-}(x) \to 0$
as $|x| \to \infty$. Suppose that $\lambda_0 = \inf \Spec(H_{\alpha}) < 0$ is an isolated
eigenvalue. Then there exists a constant $C_V$ such that for all $x \in \R^d$
$$
\varphi_0(x) \geq \frac{C_V}{(1+|x|)^{d+\alpha}}.
$$
\end{theorem}

\begin{proof}
Let first $|x|<2$. We have
$$
\varphi_0(x) \geq \inf_{y \in B(0,2)} \varphi_0(y) \geq C_{V} (1 + |x|)^{-d-\alpha}.
$$
Let now $|x| \geq 2$ and $\eta:= -2 \lambda_0 > 0$. Similarly as before, by integrating in the equality
$$
e^{- (\lambda_0 + \eta) t} \varphi_0(x) = e^{- \eta t} T_t \varphi_0(x) =
\ex^x[e_{V_{\eta}}(t)\varphi_0(X_t)],
$$
we obtain
$$
\varphi_0(x) = (\lambda_0 + \eta) G^{V_{\eta}}\varphi_0(x).
$$
Let $D:=B(x,1)$. Applying \eqref{eq:pot1} to $D' = \Rd$ and $f = \varphi_0$, and using (\ref{eq:refeta})
and \eqref{eq:IWF}, we furthermore get
\begin{align*}
\varphi_0(x)
& \geq
\ex^x[e_{V_{\eta}}(\tau_D) \varphi_0(X_{\tau_D})]
 =
C \int_D G^{V_{\eta}}_D(x,y) \int_{D^c} \varphi_0(z) |z-y|^{-d-\alpha}dzdy  \\
& \geq
C \int_D G^{V_{\eta}}_D(x,y) \int_{B(0,1)} \varphi_0(z) |z-y|^{-d-\alpha}dzdy
 \geq
C_V v_{D,\eta}(x)|x|^{-d-\alpha}.
\end{align*}
Since
$$
v_{D,\eta}(x) =
\ex^x\left[ \int_0^{\tau_D} e^{\int_0^t (V_{-}(X_s) - \eta) ds}dt\right]
\geq \ex^x\left[ \int_0^{\tau_D} e^{- \eta t}dt\right] = \frac{1-\ex^0[e^{-\eta \tau_{B(0,1)}}]}{\eta},
$$
the proof is complete.
\end{proof}

\begin{remark}
\rm{
Let $d=1$ and $\alpha \in [1,2)$. By using a martingale argument different from ours, it is possible to show
that under the same assumptions as in the theorem above $\varphi_0$ is comparable to $(1+|x|)^{-1-\alpha}$
\cite[Prop. IV.1-IV.3]{bib:CMS}.
}
\end{remark}

\begin{example}
\rm{
Let $d = 1$ and $\alpha \in [1,2)$.
\begin{itemize}
\item[(1)]
\emph{Potentials with compact support:} Let $V \not \equiv 0$ be a non-positive, bounded potential
such that $\supp V \subset [-b,b]$, where $b > 0$. Then for $x \in \R$
\begin{align*}
\label{ex:decpot}
\frac{C^{(1)}_V}{(1+|x|)^{1+\alpha}} \leq \varphi_0(x) \leq \frac{C^{(2)}_V}{(1+|x|)^{1+\alpha}}
\, .
\end{align*}
\item[(2)]
\emph{Potential well:} A special case of the above is
$$
V(x)=
\left\{
\begin{array}{ll} - a, & x \in [-b,b] \\
0,                 & x \in [-b,b]^c ,
\end{array} \right.
$$
where $a, b > 0$. Clearly, in this case the two-sided estimates in (1) above hold.
\end{itemize}
}
\end{example}

\begin{example}[\textbf{Coulomb potential}]
\rm{
A case of special interest is the semi-relativistic Coulomb potential in $d=3$, i.e., the operator
$(-\Delta + m^2)^{1/2} - m - \frac{C}{|x|}$. It is known that in the case discussed in the present paper 
(i.e. for zero particle mass $m=0$) the operator
$H_1 = \sqrt{-\Delta} - \frac{C}{|x|}$ is unbounded from below when $C>\frac{2}{\pi}$. If $C \leq
\frac{2}{\pi}$, then the operator $H_1$ is bounded from below (in fact positive), but $\Spec H_1 =
\Spec_{\rm{ess}} H_1 = [0, \infty)$ and $\inf\Spec H_1 = 0$ is not an eigenvalue (see e.g. discussion
in \cite[p.499]{bib:DaL}). Furthermore, as seen in Example \ref{ex:katoclass}, the Coulomb potential
$V(x) = - \frac{C}{|x|}$ does not belong to the fractional Kato-class $\cK^1$.
}
\end{example}

\section{Intrinsic ultracontractivity of fractional Feynman-Kac semigroups}
\subsection{Analytic and probabilistic descriptions of intrinsic ultracontractivity}
\noindent
Intrinsic ultracontractivity (IUC) has been first introduced in \cite{bib:DS} for general semigroups of
compact operators and it proved to be a strong regularity property implying a number of ``nice"
properties of operator semigroups and their spectral properties (see, for instance, \cite{bib:D}).
Important examples include semigroups of elliptic operators and Schr{\"o}dinger semigroups either on $\Rd$
or on domains $D \subset \Rd$ with Dirichlet boundary conditions \cite{bib:B, bib:D, bib:BD}. More
recently, IUC has been addressed also in the case of semigroups generated by fractional Laplacians and
fractional Schr{\"o}dinger operators on bounded domains \cite{bib:CS1, bib:CS2, bib:Ku2, bib:KaKu}.

In this section we assume that all operators $T_t$ are compact. If $V$ is non-negative, then $\lambda_0 > 0$,
however, in our case it may happen that $\lambda_0 \leq 0$.

\begin{definition}[\textbf{Intrinsic fractional Feynman-Kac semigroup}]
\rm{
Let
\begin{equation}
\widetilde{u}(t,x,y) := \frac{e^{\lambda_0 t} u(t,x,y)}{\varphi_0(x) \varphi_0(y)}.
\label{intrkernel}
\end{equation}
We call the one-parameter semigroup $\{\widetilde{T}_t: t\geq 0\}$
\begin{equation}
\label{condisemigroup}
\widetilde{T}_t f(x) = \int_{\R^d} f(y) \widetilde{u}(t,x,y) \varphi^2_0(y)dy
\end{equation}
\emph{intrinsic fractional Feynman-Kac semigroup}, acting on $L^2(\R^d,\varphi^2_0dx)$.
}
\end{definition}

\noindent
From a probabilistic point of view the intrinsic semigroup is more natural than $\{T_t: t\geq 0\}$
since for every $t>0$ and $x \in \R^d$ it has the property $\widetilde{T}_t \1_{\R^d}(x) = 1$. The
intrinsic semigroup is generated by the operator $-\widetilde{H}_{\alpha}$, where
\begin{align*}
\widetilde{H}_{\alpha} := U^{-1} (H_{\alpha} - \lambda_0)U,
\end{align*}
and where the unitary map $U: L^2(\Rd, \varphi_0^2(x)dx) \rightarrow L^2(\Rd, dx)$ is defined by
\begin{equation}
U f(x) = \varphi_0 (x) f(x).
\label{unitary}
\end{equation}
For sufficiently regular functions $f$ (e.g., from Schwartz space) this operator can be computed
explicitly to be
\begin{eqnarray}
\label{intrgen}
\widetilde{H}_{\alpha} f(x)
= \cA \int_{\R^d} \frac{f(x)-f(y)}{|y-x|^{d+\alpha}} \frac{\varphi_0(y)}{\varphi_0(x)}dy.
\nonumber
\end{eqnarray}

Intrinsic ultracontractivity originally has been defined as the property that $\widetilde{T}_t$ is a
bounded operator from $L^2(\R^d,\varphi^2_0dx)$ to $L^\infty(\Rd)$ for every $t>0$, however, for our
purposes the following equivalent definition is more suitable.
\begin{definition}[\textbf{Intrinsically ultracontractive semigroup}]
\rm{
A semigroup $\{T_t: t \geq 0\}$ is called \emph{intrinsically ultracontractive} (IUC) if for every
$t > 0$ there is a constant $C_{V,t} > 0$ such that
\begin{equation}
\tilde u(t,x,y) \le C_{V,t}
\end{equation}
for all $x,y \in \Rd$.
\label{def:IU}
}
\end{definition}
\noindent
Also, for our purposes below we propose the following property.

\begin{definition}[\textbf{Asymptotically intrinsically ultracontractive semigroup}]
\rm{
We call a semigroup $\{T_t: t\geq 0\}$ \emph{asymptotically intrinsically ultracontractive} (AIUC)
if there exists $t_0 >0$ such that for every $t \geq t_0$ there is a constant $C_{V,t}>0$ for which
\begin{equation}
\label{eq:AIU}
\tilde u(t,x,y) \le C_{V,t}
\end{equation}
for all $x,y \in \Rd$.
\label{def:AIU}
}
\end{definition}
\noindent
As it will be seen in Subsection \ref{subsec:IUC} below IUC is a stronger property than AIUC. 

\begin{remark}
\rm{
Clearly, it suffices to assume that \eqref{eq:AIU} holds for some $t_0>0$ as by the semigroup property it 
extends to all $t>t_0$. Also, it is easy to see that if $\{T_t: t\geq 0\}$ is AIUC, then there is $t_0>0$ 
such that for every $t>t_0$ and all $x,y\in \Rd$
\begin{equation}
\label{eq:belowIU}
\tilde u(t,x,y) \geq C^{(1)}_{V,t}
\end{equation}
with a constant $C^{(1)}_{V,t} > 0$. The same applies for IUC, i.e., a lower bound holds for every $t>0$.
An immediate consequence of this is that if the semigroup is AIUC, then $\varphi_0 \in L^1(\R^d)$.
}
\end{remark}

\begin{lemma}
\label{lm:iucons}
The following two conditions are equivalent.
\begin{itemize}
\item[(1)]
The semigroup $\{T_t: t\geq 0\}$ is AIUC.
\item[(2)]
The property
\begin{align}
\label{eq:iucons}
\widetilde{u}(t,x,y) \stackrel {t \rightarrow \infty} \longrightarrow 1,
\end{align}
holds, uniformly in $(x,y) \in \R^d \times \R^d$.
\end{itemize}
\end{lemma}

\begin{proof}
The implication $(2) \Rightarrow (1)$ is immediate, we only show the converse statement. We have for every
$x,y \in \R^d$ and $t>2t_0$
\begin{align*}
& \left|\widetilde{u}(t,x,y)-1\right| \\
& = \left|\int_{\R^d} \int_{\R^d} \frac{ u(t_0,x,z)u(t-2t_0,z,w)u(t_0,w,y)}{e^{- \lambda_0 t}\varphi_0(x) \varphi_0(y)} dzdw
- \frac{ e^{- \lambda_0 t}\varphi_0(x) \varphi_0(y)}{e^{- \lambda_0 t}\varphi_0(x) \varphi_0(y)}\right| \\
& = \left|\int_{\R^d} \int_{\R^d} \frac{ u(t_0,x,z)\varphi_0(z)\left(u(t-2t_0,z,w) - e^{- \lambda_0 (t-2t_0) }\varphi_0(z)
\varphi_0(w)\right)u(t_0,w,y)\varphi_0(w)}{e^{- \lambda_0 t}\varphi_0(x)\varphi_0(z)
\varphi_0(w)\varphi_0(y)} dzdw\right| \\
& \leq e^{\lambda_0 t} \left\|\frac{u(t_0,x,y)}{\varphi_0(x)\varphi_0(y)}\right\|^2_\infty\int_{\R^d}\int_{\R^d}
\left|u(t-2t_0,z,w) - e^{- \lambda_0 (t-2t_0) }\varphi_0(z) \varphi_0(w)\right|\varphi_0(z)\varphi_0(w) dzdw \\
& \leq C e^{\lambda_0 t} \left(\int_{\R^d} \int_{\R^d} \left|u(t-2t_0,z,w) -
e^{- \lambda_0 (t-2t_0) }\varphi_0(z) \varphi_0(w)\right|^2 dzdw\right)^{1/2}
\, .
\end{align*}
The last factor on the right hand side is the Hilbert-Schmidt norm of the operator $T_{t-2t_0} -
e^{-\lambda_0 (t-2t_0)}P_{\varphi_0}$, where $P_{\varphi_0}:L^2(\R^d) \rightarrow L^2(\R^d)$ is the projection
onto the one dimensional subspace of $L^2(\R^d)$ spanned by $\varphi_0$. This gives
$$
\left|\widetilde{u}(t,x,y)-1\right| \leq C e^{\lambda_0 t }
\left( \sum_{k=1}^\infty e^{-2\lambda_k (t-2t_0) }\right)^{1/2}
= C e^{2t_0\lambda_1} e^{-(\lambda_1- \lambda_0) t}
\left( \sum_{k=1}^\infty e^{-2(\lambda_k - \lambda_1) (t-2t_0) }\right)^{1/2}.
$$
By dominated convergence the last sum converges to the multiplicity of $\lambda_1$ as $t \rightarrow \infty$.
Since $\lambda_1 >\lambda_0$, \eqref{eq:iucons} follows.
\end{proof}

In Section 5 below it will be seen that (A)IUC has a direct impact on the properties of stationary Gibbs measures
of stable processes under Kato-decomposable potentials. To obtain information on the structure of these measures
(such as typical sample path behaviour and fluctuations) it is useful to understand IUC and AIUC in an alternative
probabilistic way on the level of the semigroup $\{T_t: t \geq 0\}$.

For the remainder of this section we will use the following conditions.
\begin{assumption}
\rm{
Suppose that $V$ is a Kato-decomposable potential such that for every $t>0$ the operators $T_t$ are compact.
Moreover, let
\begin{align}
\label{eq:condition2}
T_t \1_{\R^d}(x) \leq C_{D,t} T_t \1_D (x)
\, ,
\end{align}
where $t>0$, $x \in \R^d$, $D$ is a bounded non-empty Borel subset of $\R^d$, and $C_{D,t}>0$. We will
consider the following assumptions.
\begin{enumerate}
\item[(1)]
\label{cond:comp11}
For every $t>0$ there exists $D$ and $C_{D,t}$ such that \eqref{eq:condition2} holds for
all $x \in \R^d$.
\item[(2)]
\label{cond:comp12}
For every $t>0$ and $D$ there exists $C_{D,t}$ such that \eqref{eq:condition2} holds for
all $x \in \R^d$.
\item[(3)]
\label{cond:comp21}
There exists $t_0>0$ such that for every $t>t_0$ there is $D$ and $C_{D,t}$ such that
\eqref{eq:condition2} holds for all $x \in \R^d$.
\item[(4)]
\label{cond:comp22}
There exists $t_0>0$ such that for every $t>0$ and every $D$ there is $C_{D,t}$ such that
\eqref{eq:condition2} holds for all $x \in \R^d$.
\end{enumerate}
}
\end{assumption}
\noindent
Clearly, by the semigroup property $T_t T_s = T_{t+s}$ whenever \eqref{eq:condition2} holds for some
$t>0$, set $D$ and constant $C_{D,t}$, then it holds for all $s \geq t$ with the same $D$ and $C_{D,t}$.

First we note that IUC can be characterized by the above conditions.

\begin{lemma}
Let Assumption \ref{cond:comp11} (1) hold. Then the semigroup $\{T_t: t \geq 0\}$ is IUC. Let the semigroup
$\{T_t: t \geq 0\}$ be IUC. Then Assumption \ref{cond:comp12} (2) holds.
\end{lemma}

\begin{proof}
First assume that the semigroup $\{T_t: t\geq 0\}$ is IUC. Fix $t > 0$ and a bounded set $D \subset \R^d$.
For $x \in \R^d$ we have
$$
T_t \1_{\R^d} (x) = \int_{\R^d} u(t,x,y) dy \leq C_{V,t} \left\| \varphi_0 \right\|_1 \varphi_0(x).
$$
On the other hand,
$$
T_t \1_D(x) = \int_D u(t,x,y) dy \geq C_{V,t} \varphi_0(x) \int_D \varphi_0(y) dy
$$
and Assumption \ref{cond:comp12} (2) follows.

Let now Assumption \ref{cond:comp11} (1) be satisfied. For every $x,y \in \R^d$ and $t>0$ by the semigroup
property
\begin{eqnarray*}
u(t, x, y)
&=&
\int_{\R^d} \int_{\R^d} u(t/3,x,z) u(t/3,z,w) u(t/3,w,y) dzdw \\
&\leq&
C_{V,t} T_{t/3} \1_{\R^d} (x) T_{t/3} \1_{\R^d} (y) \leq C_{V,t} T_{t/3} \1_D (x) T_{t/3} \1_D (y) \\
&\leq&
\frac{C_{V,t}}{(\inf_{y \in D} \varphi_0(y))^2} T_{t/3} \varphi_0(x) T_{t/3} \varphi_0(y) =
C_{V,t} e^{-2 \lambda_0 t/3} \varphi_0(x) \varphi_0(y).
\end{eqnarray*}
\end{proof}
\noindent
Conditions (1)-(2) above were used also in \cite{bib:KS} in proving IUC of the relativistic $\alpha$-stable
Feynman-Kac semigroup. A straightforward corollary of the above lemma is the following.
\begin{corollary}
Consider the semigroup $\{T_t: t\geq 0\}$.
\begin{enumerate}
\item[(1)]
If Assumption \ref{cond:comp21} (3) holds, then $\{T_t: t\geq 0\}$ is AIUC. If $\{T_t: t\geq 0\}$ is AIUC,
then Assumption \ref{cond:comp22} (4) holds.
\item[(2)]
The semigroup $\{T_t: t\geq 0\}$ is IUC if and only if either of the two equivalent Assumptions
\ref{cond:comp11} (1) and \ref{cond:comp12} (2) is satisfied, and it is AIUC if and only if either of the
two equivalent Assumptions \ref{cond:comp21} (3) and \ref{cond:comp22} (4) holds.
\end{enumerate}
\end{corollary}

Using the above statements we can give an equivalent probabilistic definition of IUC and AIUC.
\begin{definition}
\rm{
Let $V$ be a Kato-decomposable potential. We say that the corresponding semigroup $\{T_t: t\geq 0\}$ is
\emph{intrinsically ultracontractive (IUC)} whenever for every $t>0$ there exist a non-empty bounded
Borel set $D \subset \R^d$ and a constant $C_{V,t} > 0$ such that for all $x \in \R^d$
\begin{align}
\label{cond:probIU}
\ex^x\left[X_t \in D^c; e_V(t) \right] \leq C_{V,t} \: \ex^x\left[X_t \in D; e_V(t) \right]
\end{align}
holds. We say that the semigroup $\{T_t: t\geq 0\}$ is \emph{asymptotically intrinsically ultracontractive
(AIUC)} whenever there exists $t_0 > 0$ such that for every $t \geq t_0$ there is a non-empty bounded set
$D \subset \R^d$ and a constant $C_{V,t}>0$ such that for all $x \in \R^d$ inequality \eqref{cond:probIU}
holds.
}
\end{definition}

\subsection{Ultracontractivity properties of intrinsic fractional Feynman-Kac semigroups}
\label{subsec:IUC}
\noindent
Our main goal here is to establish and characterize IUC and AIUC for fractional Schr\"odinger operators with
Kato-decomposable potentials. While IUC usually is defined and considered for non-negative potentials, we do
not assume positivity and include also the case when the bottom of the spectrum may be negative.

First we need the following technical lemma.
\begin{lemma}
\label{lm:techiu}
Let $V$ be a Kato-decomposable potential and $D \subset \Rd$ be an arbitrary open set. Then for every
$t>0$ we have that
\begin{itemize}
\item [(1)]
$\ex^x\left[\frac{t}{2} \geq \tau_D; e_V(t)\right] \leq C_{V,t}
\ex^x\left[e_V(\tau_D)T_{\frac{t}{2}}\1(X_{\tau_D})\right]$
\item[(2)]
$\ex^x\left[\frac{t}{2}<\tau_D; e_V(t)\right] \leq C_{V,t}
\ex^x\left[\frac{t}{4} < \tau_D; e_V \left(\frac{t}{4}\right)\right] \sup_{y \in D}T_{3t/4}\1(y)$.
\end{itemize}
\end{lemma}

\begin{proof}
By the plain and the strong Markov properties we obtain
\begin{eqnarray*}
\lefteqn{
\ex^x \left[\frac{t}{2} \geq \tau_D; e_{V_{+}}(t)e_{-V_{-}}(t)\right] } \\
&& \leq
\ex^x \left[\frac{t}{2}\geq\tau_D;e_V(\tau_D)e^{-\int_{\tau_D}^{\frac{t}{2}+\tau_D}V_{+}(X_s)ds}
e^{\int_{\tau_D}^{t+\tau_D}V_{-}(X_s) ds}\right] \\
&& \leq
\ex^x\left[e_V(\tau_D)\ex^{X_{\tau_D}} \left[e_{V_{+}}\left(\frac{t}{2}\right)e_{-V_{-}}(t)\right]\right]\\
&& =
\ex^x\left[e_V(\tau_D)\ex^{X_{\tau_D}} \left[e_{V_{+}}\left(\frac{t}{2}\right)e_{-V_{-}}
\left(\frac{t}{2}\right)\right. \left.\ex^{X_{t/2}}\left[e_{-V_{-}}\left(\frac{t}{2}\right)\right]\right]
\right]\\
&& \leq
\sup_{y \in \Rd} \ex^y\left[e_{-V_{-}}\left(\frac{t}{2}\right)\right] \ex^x\left[e_V(\tau_D)\ex^{X_{\tau_D}} \left[e_{V_{+}}\left(\frac{t}{2}\right)e_{-V_{-}}\left(\frac{t}{2}\right) \right]\right]\\
&& \leq
C_{V,t} \ex^x\left[e_V(\tau_D)\ex^{X_{\tau_D}} \left[e_V\left(\frac{t}{2}\right) \right]\right].
\end{eqnarray*}
This gives (1). Similarly, once again by the Markov property
\begin{align*}
\ex^x\left[\frac{t}{2}<\tau_D; e_V(t)\right]
& =
\ex^x\left[\frac{t}{4} < \tau_D; e_V\left(\frac{t}{4}\right)\ex^{X_{t/4}}
\left[\frac{t}{4}< \tau_D; e_V\left(\frac{3t}{4} \right)\right]\right] \\
& \leq
\sup_{y \in D}T_{3t/4}\1(y) \: \ex^x\left[\frac{t}{4} < \tau_D; e_V\left(\frac{t}{4}\right)\right],
\end{align*}
which completes the proof.
\end{proof}

For the remainder of this section we will use the following conditions.

\begin{assumption}
\label{ass:AIUC}
\rm{
Let $V$ be a Kato-decomposable potential such that $V(x) \rightarrow \infty$ as $|x| \rightarrow \infty$.
Consider the following assumptions.
\begin{itemize}
\item[(1a)]
For any $t > 0$ there is a constant $C_{V,t}>0$ such that for all $x,y \in \R^d$
\begin{align}
\label{as:ass1}
u(t,x,y) \leq C_{V,t} (1+|x|)^{-d-\alpha} (1+|y|)^{-d-\alpha}.
\end{align}
\item[(1b)]
There exists $t_0>0$ such that for any $t > t_0$ there is a constant $C_{V,t}>0$ such that for all
$x,y \in \R^d$ \eqref{as:ass1} holds.
\item[(2a)]
For any $t > 0$ there is a constant $C_{V,t}>0$ such that for all $r > 0$, $x \in \overline{B}(0,r)^c$
\begin{align}
\label{as:ass2}
\ex^x[t<\tau_{\overline{B}(0,r)^c}; e_V(t)] \leq C_{V,t} (1+r)^{-d-\alpha}.
\end{align}
\item[(2b)]
There exists $t_0>0$ such that for any $t > t_0$ there is a constant $C_{V,t}>0$ such that for all $r > 0$,
$x \in \overline{B}(0,r)^c$ \eqref{as:ass2} holds.
\item[(3a)]
For any $t > 0$ there is a constant $C_{V,t}>0$ such that for all $x \in \R^d$
\begin{align}
\label{as:ass3}
T_t \1(x) \leq C_{V,t} (1+|x|)^{-d-\alpha}.
\end{align}
\item[(3b)]
There exists $t_0>0$ such that for any $t > t_0$ there is a constant $C_{V,t}>0$ such that for all $x \in \R^d$
\eqref{as:ass3} follows.
\end{itemize}
}
\end{assumption}

Our first main characterization result is as follows.

\begin{theorem}[\textbf{Characterization of IUC and AIUC}]
\label{th:charIUC}
Let $V$ be a Kato-decomposable potential such that $V(x) \rightarrow \infty$ as $|x| \rightarrow \infty$.
\begin{enumerate}
\item[(1)]
The semigroup $\{T_t: t\geq 0\}$ is intrinsically ultracontractive if and only if any of the three equivalent
conditions $(1a)$, $(2a)$, $(3a)$ in Assumption \ref{ass:AIUC} hold.
\item[(2)]
The semigroup $\{T_t: t\geq 0\}$ is asymptotically intrinsically ultracontractive if and only if any of the
three equivalent conditions $(1b)$, $(2b)$ and $(3b)$ in Assumption \ref{ass:AIUC} is satisfied.
\end{enumerate}
\end{theorem}

\begin{proof}
We only prove the equivalence of IUC with conditions $(1a)$, $(2a)$ and $(3a)$; the proof of equivalence of 
AIUC with $(1b)$, $(2b)$ and $(3b)$ can be done in the same way. We proceed in a succession of steps.

\smallskip
\noindent
\emph{(Step 1)}
For the proof of the implication IUC $\Rightarrow$ (1a) consider the set
$$
A = \left\{x \in \R^d: B(x,1) \cap \supp(V_{-}) = \emptyset \right\}.
$$
Clearly, by the assumption on the potential $A^c$ is bounded and $V\geq0$ on each $B(x,1)$ for $x \in A$.
If $x,y \in A$, then (1a) follows by the definition of IUC and the upper bound in Theorem \ref{th:eig}.
Whenever $x,y \in A^c$, then the boundedness of $u(t,x,y)$ and $A^c$ give (1a). If now $x \in A$, $y \in A^c$,
then we have
$$
u(t,x,y) \leq C_{V,t} \varphi_0(x) \varphi_0(y) \leq C_{V,t} \varphi_0(x) \leq C_{V,t} (1+|x|)^{-d-\alpha}
(1+|y|)^{-d-\alpha}
$$
by an argument similar as above. The case $x \in A^c$, $y \in A$ follows by symmetry.

\medskip
\noindent
\emph{(Step 2)}
By (1a) we have
\begin{align*}
\ex^x[t<\tau_{\overline{B}(0,r)^c}; e_V(t)]
& \leq
\ex^x[X_t \in \overline{B}(0,r)^c; e_V(t)] \\
& =
\int_{\overline{B}(0,r)^c} u(t,x,y) dy \leq C_{V,t} (1+|x|)^{-d-\alpha} \leq C_{V,t} (1+r)^{-d-\alpha}
\,  ,
\end{align*}
for $x \in \overline{B}(0,r)^c$. This gives (2a).

\medskip
\noindent
\emph{(Step 3)}
Next we prove (2a) $\Rightarrow$ (3a). Let $R > 1$ be sufficiently large so that $V(y) \geq 1$
for $|y| \geq R$. Let $|x| \geq 2R$, $r=|x|/2$ and $D = B(x,r)$. It is clear that $D \cap \supp(V_{-})
= \emptyset$. We write
$$
T_t\1(x)=\ex^x\left[\frac{t}{2}<\tau_D; e_V(t)\right] + \ex^x\left[\frac{t}{2} \geq \tau_D; e_V(t)\right].
$$
By condition (2a) and Lemma \ref{lm:techiu} we obtain
\begin{align*}
\ex^x\left[\frac{t}{2}<\tau_D; e_V(t)\right]
& \leq
C_{V,t} \ex^x\left[\frac{t}{4}<\tau_D; e_V\left(\frac{t}{4}\right)\right] \\
& \leq
C_{V,t} \ex^x\left[\frac{t}{4}<\tau_{\overline{B}(0,r)^c}; e_V\left(\frac{t}{4}\right)\right]
\leq
C_{V,t} (1 + |x|)^{-d-\alpha} \,
\end{align*}
and
$$
\ex^x\left[\frac{t}{2} \geq \tau_D; e_V(t)\right] \leq
C_{V,t} \ex^x\left[e_V(\tau_D)T_{t/2} \1(X_{\tau_D})\right].
$$
Thus
\begin{align}
\label{eq:Tt}
T_t\1(x) \leq C_{V,t} \left((1 + |x|)^{-d-\alpha} + \ex^x\left[e_V(\tau_D)T_{t/2} \1(X_{\tau_D})\right]\right).
\end{align}
We need to estimate the latter expectation. Put
\begin{equation*}
\label{def2}
f(y) =
\begin{cases}
\ex^y\left[e_V(\tau_D) T_{t/2}\1(X_{\tau_D})\right] & \text{for} \quad y \in D,\\
T_{t/2}\1(y) &  \text{for} \quad y \in D^c.   	
\end{cases}
\end{equation*}
Then $f(y) = \ex^y\left[e_V(\tau_D) f(X_{\tau_D})\right]$, $y \in D$, and by \eqref{eq:bhisc} we obtain
\begin{equation}
\begin{split}
\label{eq:f}
f(x)
& \leq C
\int_{B(x,r/2)^c}\frac{f(y)}{|y-x|^{d+\alpha}}dy  \\
& =
C \left(\int_{D \backslash B(x,r/2)}\frac{\ex^y\left[e_V(\tau_D) T_{t/2}\1(X_{\tau_D})\right]}
{|y-x|^{d+\alpha}}dy + \int_{D^c}\frac{T_{t/2}\1(y)}{|y-x|^{d+\alpha}}dy\right).
\end{split}
\end{equation}
Hence by a combination of \eqref{eq:Tt} and \eqref{eq:f}
\begin{equation}
\begin{split}
\label{eq:Tt1}
T_t\1(x)
& \leq
C_{V,t} \left((1 + |x|)^{-d-\alpha} + \int_{B(x,r/2)^c}\frac{T_{t/2}\1(y)}{|y-x|^{d+\alpha}}dy
\right.\\
& \left.
\qquad \qquad +
\sup_{y \in D}\ex^y\left[e_V(\tau_D) T_{t/2}\1(X_{\tau_D})\right] (1+|x|)^{-\alpha} \right)
\end{split}
\end{equation}
is obtained. The fact that $V \geq 1$ on $D$ and \eqref{eq:IWF} imply for $y \in D$ that
\begin{align*}
\ex^y\left[e_V(\tau_D) T_{t/2}\1(X_{\tau_D})\right]
& =
\ex^y\left[X_{\tau_D} \in B(x,3r/2) \backslash D; e_V(\tau_D) T_{t/2}\1(X_{\tau_D})\right] \\
& \qquad +
\ex^y\left[X_{\tau_D} \in B(x,3r/2)^c; e_V(\tau_D) T_{t/2}\1(X_{\tau_D})\right] \\
& \leq
u_D(y) \sup_{z \in B(x,3r/2)} T_{t/2}\1(z)+\cA\int_D G^V_D(y,z)\int_{B(x,3r/2)^c}
\frac{T_{t/2}\1(v)}{|v-z|^{d+\alpha}}dvdz \\
& \leq
u_D(y) \sup_{z \in B(x,3r/2)} T_{t/2}\1(z) + Cv_D(y) \int_{B(x,3r/2)^c}
\frac{T_{t/2}\1(v)}{|v-x|^{d+\alpha}}dv \\
& \leq
\sup_{z \in B(x,3r/2)} T_{t/2}\1(z) + C \int_{B(x,r/2)^c}
\frac{T_{t/2}\1(y)}{|y-x|^{d+\alpha}}dy.
\end{align*}
Thus we obtain from \eqref{eq:Tt1}
\begin{equation}
\begin{split}
\label{eq:Tt3}
T_t\1(x)
\leq
C_{V,t} \left((1 + |x|)^{-d-\alpha} + \int_{B(x,r/2)^c}\frac{T_{t/2}\1(y)}{|y-x|^{d+\alpha}}dy
\; + \right.
\left.\sup_{z \in B(x,3r/2)} T_{t/2}\1(z) (1+|x|)^{-\alpha} \right).
\end{split}
\end{equation}
Suppose now that for some $\gamma \geq 0$, $\gamma \neq d$, we have $T_t\1(x) \leq C_{V,t,\gamma}
(1+|x|)^{-\gamma}$, for all $x \in \R^d$, $t>0$. This clearly holds for $\gamma = 0$. Then by \eqref{eq:Tt3}
and \eqref{eq:cruest} we obtain
\begin{equation}
\begin{split}
\label{eq:Tt2}
T_t\1(x) & \leq C_{V,t} (1 + |x|)^{-d-\alpha} + C_{V,t,\gamma} (1 + |x|)^{-\gamma-\alpha} \\
&
\qquad + C_{V,t,\gamma} \int_{B(x,r/2)^c}(1 + |y|)^{-\gamma}|y-x|^{-d-\alpha}dy
\leq
C_{V,t,\gamma}(1 + |x|)^{-\gamma^{'}}
\end{split}
\end{equation}
for $\gamma^{'} = \min(\gamma +\alpha, d+\alpha)$ and $|x| \geq 2R$. Also, $T_t\1(x) \leq
C_{V,t,\gamma}(1 + |x|)^{-\gamma^{'}}$ for $|x| \leq 2R$.

Now, starting from \eqref{eq:Tt3} again and taking $\gamma = \gamma^{'}$ in \eqref{eq:Tt2}, we obtain the
bounds \eqref{eq:Tt2} with larger $\gamma^{'}$. By using this argument recursively, we can improve the order
of the estimate $T_t\1(x) \leq C_{V,t,\gamma}(1 + |x|)^{-\gamma^{'}}$. If $\gamma^{'} = d$ occurs after some
step, then we take $\gamma = d - \frac{\alpha}{2}$ in the next one. On iteration, after $\left\lfloor 2 +
\frac{d}{\alpha}\right\rfloor$ steps $T_t\1(x) \leq C_{V,t}(1 + |x|)^{-d-\alpha}$ follows, for all $x \in \R^d$.

\medskip
\noindent
\emph{(Step 4)}
To complete the proof of the theorem we prove the implication (3a) $\Rightarrow$ IUC. By the bound
\begin{align*}
u(t,x,y) =
\int_{\R^d} \int_{\R^d} u\left(\frac{t}{3},x,z\right)u\left(\frac{t}{3},z,v\right)u\left(\frac{t}{3},v,y\right)dvdz
\leq C_{V,t} T_{t/3}\1(x) T_{t/3}\1(y),
\end{align*}
it suffices to show that $T_t\1(x) \leq C_{V,t} \varphi_0(x)$, for $x \in \R^d$ and $t > 0$.
Put $D=B(x,1)$ and $r = \frac{|x|}{2}$. Let $R>3$ be sufficiently large so that $D \cap \supp(V_{-}) = \emptyset$
for $|x| > R$. If $\lambda_1 > 0$, we choose $\eta=0$, if $\lambda_1 = 0$, we choose $\eta=1$, and if $\lambda_1 < 0$
we choose $\eta = -2 \lambda_1$. In all these cases $\eta + \lambda_1 > 0$. Then we have
\begin{align}
\label{eq:sem}
T_t\1(x) = e^{\eta t} e^{-\eta t} T_t\1(x) = C_{V,t} \left(\ex^x\left[\frac{t}{2}<\tau_D; e_{V_{\eta}}(t)\right]
+ \ex^x\left[\frac{t}{2} \geq \tau_D; e_{V_{\eta}}(t)\right]\right) \, ,
\end{align}
where $V_{\eta} = V + \eta$. We start by estimating the first expectation in \eqref{eq:sem}. Note that
\begin{align*}
v_{D, \eta}(x)
& =
\ex^x\left[\int_0^{\tau_D} e^{-\int_0^v V_{\eta}(X_s)ds}dv\right] \geq
\ex^x\left[\frac{t}{4} < \tau_D;\int_0^{\frac{t}{4}} e^{-\int_0^v V_{\eta}(X_s)ds}dv \right] \\
& \geq
\ex^x\left[\frac{t}{4} < \tau_D; \frac{t}{4}e^{-\int_0^{\frac{t}{4}}V_{\eta}(X_s)ds}\right]  =
\frac{t}{4} \ex^x\left[\frac{t}{4} < \tau_D; e_{V_{\eta}}\left(\frac{t}{4}\right)\right] \, .
\end{align*}
Using this, Lemma \ref{lm:techiu} (2) and condition (3a), we obtain
\begin{align}
\label{eq:sem1}
\ex^x\left[\frac{t}{2}<\tau_D; e_{V_{\eta}}(t)\right] \leq C_{V,t} \ex^x\left[\frac{t}{4} < \tau_D; e_{V_{\eta}}
\left(\frac{t}{4}\right)\right] \sup_{y \in D}T_{3t/4}\1(y) \leq C_{V,t} v_{D,\eta}(x) (1 + |x|)^{-d-\alpha}
\, .
\end{align}

For the second expectation in \eqref{eq:sem} a combination of Lemma \ref{lm:techiu} (1), \eqref{eq:IWF},
\eqref{eq:est}, condition (3a) and \eqref{eq:cruest} yields
\begin{align*}
\ex^x
& \left[\frac{t}{2} \geq \tau_D;  e_{V_{\eta}}(t)\right]  \leq
C_{V,t} \ex^x \left[e_{V_{\eta}}(\tau_D)\ex^{X_{\tau_D}} \left[e_{V_{\eta}}\left(\frac{t}{2}\right)\right]\right] \\
& = C_{V,t} \ex^x \left[X_{\tau_D} \in B(x,r); e_{V_{\eta}}(\tau_D)\ex^{X_{\tau_D}}
\left[e_{V_{\eta}}\left(\frac{t}{2}\right)\right]\right] \\
& \qquad + C_{V,t} \ex^x \left[X_{\tau_D} \in B(x,r)^c; e_{V_{\eta}}(\tau_D)\ex^{X_{\tau_D}}
\left[e_{V_{\eta}}\left(\frac{t}{2}\right)\right]\right] \\
& \leq C_{V,t}\left( u_{D,\eta}(x) \sup_{y \in B(x,r)} T_{t/2}\1(y) +
\int_D G^{V_\eta}_D(x,y) \int_{B(x,r)^c} T_{t/2}\1(z) |z-y|^{-d-\alpha}dzdy\right) \\
& \leq C_{V,t} \left(v_{D,\eta}(x)(1 + |x|)^{-d-\alpha} + v_{D,\eta}(x) \int_{B(x,r)^c}
(1 + |z|)^{-d-\alpha} |z-x|^{-d-\alpha}dz\right) \\
& \leq C_{V,t}v_{D,\eta}(x)(1 + |x|)^{-d-\alpha} \, .
\end{align*}
By \eqref{eq:sem} and \eqref{eq:sem1} this gives  $T_t\1(x) \leq C_{V,t} v_{D,\eta}(x) (1 + |x|)^{-d-\alpha}$
for $|x| > R$. Thus by Theorem \ref{th:eig} we obtain $T_t\1(x) \leq C_{V,t} \varphi_0(x)$ for $|x| > R$. Since
$\varphi_0$ is continuous and strictly positive, we have that $\inf_{z \in B(0,R)} \varphi_0(z)>0$. Hence for
$|x| \leq R$ we have
$$
T_t\1(x) \leq C_{V,t} \inf_{z \in B(0,R)} \varphi_1(z) \leq C_{V,t} \varphi_0(x),
$$
which completes the proof of the theorem.
\end{proof}

Using Theorem \ref{th:charIUC} a sufficient condition for (A)IUC in terms of the behaviour of the potential $V$ at
infinity is as follows.

\begin{theorem}[\textbf{Sufficient condition for IUC and AIUC}]
\label{th:suff}
Let $V$ be a Kato-decomposable potential. Then:
\begin{enumerate}
\item[(1)] If there exists $R>1$ and $C_{V,R}>0$ such that for all $|x|>R$
\begin{equation}
\label{eq:suffAIU}
\frac{V(x)}{\log |x|} \geq C_{V,R},
\end{equation}
then each operator $T_t$, $t>0$, is compact and the semigroup $\{T_t: t\geq 0\}$  is asymptotically
intrinsically ultracontractive.
\item[(2)] If moreover
$$
\lim_{|x| \rightarrow \infty}\frac{V(x)}{\log |x|} = \infty,
$$
then the semigroup $\{T_t: t\geq 0\}$  is intrinsically ultracontractive.
\end{enumerate}
\end{theorem}

\begin{proof}
Denote $g(r) = \inf_{x \in B(0,r)^c} V(x)$. We have
\begin{align*}
\ex^x[t<\tau_{\overline{B}(0,r)^c}; e_V(t)] & \leq e^{-g(r)t},
\end{align*}
for every $x \in \overline{B}(0,r)^c$, $r>0$.

First we prove (1). By condition \eqref{eq:suffAIU} we have $\lim_{|x| \rightarrow \infty}V(x) = \infty$ and by Lemma
\ref{lm:compact} each $T_t$ is compact. Let $r>R$. Fix $t_0 = \frac{\alpha+d}{C_{V,R}}$. By assumption, for all $t \geq t_0$ we have
$$
g(r) \geq C_{V,R} \log(r) \geq \frac{d+\alpha}{t} \log r,
$$
which gives
$$
e^{-g(r)t} \leq C (1 + r)^{-d-\alpha}
$$
for $r>R$. We obtain
\begin{align*}
\ex^x[t<\tau_{\overline{B}(0,r)^c}; e_V(t)] \leq C (1 + r)^{-d-\alpha},
\end{align*}
for every $x \in \overline{B}(0,r)^c$, $r>R$, and $t \geq t_0 = \frac{\alpha+d}{C_{V,R}}$.

If $r \leq R$ and $x \in \overline{B}(0,r)^c$, then
$$
\ex^x[t<\tau_{\overline{B}(0,r)^c}; e_V(t)] \leq C_{V,t} = C_{V,t}(1+R)^{d+\alpha} (1+R)^{-d-\alpha} \leq
C_{V,t} (1+r)^{-d-\alpha}.
$$
Hence there exists $t_0 >0$ such that for $t \geq t_0$ and $x \in \overline{B}(0,r)^c$, $r>0$, we have
\begin{align*}
\ex^x[t<\tau_{\overline{B}(0,r)^c}; e_V(t)] & \leq  C_{V,t} (1+r)^{-d-\alpha}
\, .
\end{align*}
This is the condition (2b) in Assumption \ref{ass:AIUC} and the assertion follows now by Theorem \ref{th:charIUC}.

For the proof of (2) observe that by the assumption for every $t>0$ there is $R>0$ such that $g(r) \geq 
\frac{d+\alpha}{t} \log(1+r)$, for $r>R$. This leads us to the condition (2a) in Assumption \ref{ass:AIUC} 
in a similar way as before and the assertion again follows by Theorem \ref{th:charIUC}.
\end{proof}

\begin{theorem}[\textbf{Necessary condition for IUC and AIUC}]
\label{th:nec}
Let $V$ be a Kato-decomposable potential such that $V(x) \rightarrow \infty$ as $|x| \rightarrow \infty$. 
\begin{enumerate}
\item [(1)] If the
semigroup $\{T_t: t\geq 0\}$  is intrinsically ultracontractive, then for any $\varepsilon \in (0,1]$
$$
\lim_{|x| \rightarrow \infty}\frac{\sup_{y \in B(x,\varepsilon)} V(y)}{\log |x|} = \infty.
$$
\item[(2)] If the semigroup $\{T_t: t\geq 0\}$ is asymptotically intrinsically ultracontractive, then there exists a
constant $C_V>0$ such that for every $\varepsilon \in (0,1]$ there is $R_{\varepsilon}>2$ such that for
all $|x|>R_{\varepsilon}$
\begin{equation}
\label{eq:necAIU}
\frac{\sup_{y \in B(x,\varepsilon)} V(y)}{\log |x|} \geq C_{V}.
\end{equation}
\end{enumerate}
\end{theorem}

\begin{proof}
Set $r = \frac{|x|}{2}$ for $|x| \geq 2$ and $D=B(x,\varepsilon)$ for an arbitrary $0< \varepsilon \leq 1$. First we prove (1). By Theorem \ref{th:charIUC} the condition (2a) in Assumption \ref{ass:AIUC} follows. Then we have for $|x| \geq 2$ and $t > 0$ that
\begin{align*}
\pr^x(t<\tau_D) e^{-\sup_{y \in D} V(y) t} \leq \ex^x[t<\tau_D; e_V(t)] \leq
\ex^x[t<\tau_{\overline{B}(0,r)^c}; e_V(t)]  \leq  C_{V,t} (1 + r)^{-d-\alpha}.
\end{align*}
Hence for $0< t \leq 1$ and $|x| \geq 2$,
\begin{align*}
\pr^0(1<\tau_{B(0,\varepsilon)}) e^{-\sup_{y \in D} V(y) t} \leq C_{V,t} |x|^{-d-\alpha} \, .
\end{align*}
It follows that $e^{-\sup_{y \in D} V(y) t} \leq C_{V,t,\varepsilon} |x|^{-d-\alpha}$ and thus
\begin{align*}
\frac{\sup_{y \in D} V(y)}{\log|x|} \geq \frac{\alpha +d}{t} - \frac{C_{V,t,\varepsilon}}{t \log|x|}.
\end{align*}
This implies $\liminf_{|x| \rightarrow \infty} \frac{\sup_{y \in D} V(y)}{\log|x|} \geq
\frac{\alpha +d}{t}$, for any $0 < t \leq 1$.

For the proof of (2) observe that by using Theorem \ref{th:charIUC} and the condition (2b) in Assumption \ref{ass:AIUC}, similarly as before, we have for $|x| \geq 2$,
\begin{align*}
\pr^0(t_0<\tau_{B(0,\varepsilon)}) e^{-\sup_{y \in D} V(y) t_0} \leq C_{V,t_0} |x|^{-d-\alpha} \, .
\end{align*}
It follows that
\begin{align*}
e^{-\sup_{y \in D} V(y) t_0} \leq \frac{C_{V,t_0}}{\pr^0(t_0<\tau_{B(0,\varepsilon)})} |x|^{-d-\alpha} \,
\end{align*}
and thus
$$
\frac{\sup_{y \in D} V(y)}{\log|x|} \geq \frac{1}{t_0}\left(\alpha +d - \frac{\log \left(\frac{C_{V,t_0}}{\pr^0(t_0<\tau_{B(0,\varepsilon)})}\right)}{\log|x|}\right).
$$
Now it is enough to choose $R_{\varepsilon} > 2$ such that for $|x|>R_{\varepsilon}$ we have
$$
\frac{\alpha +d}{2} \geq \frac{\log \left(\frac{C_{V,t_0}}{\pr^0(t_0<\tau_{B(0,\varepsilon)})}\right)}{\log|x|}.
$$
\end{proof}

For potentials $V$ comparable on unit balls outside a compact set we obtain the following result.
\begin{corollary}[\textbf{Borderline case}]
Let $V$ be a Kato decomposable potential such that $V(x) \rightarrow \infty$ as $|x| \rightarrow \infty$.
Suppose there exist $R>1$ such that $B(0,R-1)^c \cap \supp(V_{-}) = \emptyset$, and a constant $M_V>0$ such
that for every $|x| > R$ and $y \in B(x,1)$
\begin{align}
\label{eq:Mq}
V(y) \leq M_V \: V(x)
\end{align}
holds. Then:
\begin{enumerate}
\item[(1)] The semigroup $\{T_t: t\geq 0\}$ is IUC if and only if
$$
\lim_{|x| \rightarrow \infty}
\frac{V(x)}{\log |x|} = \infty.
$$
\item[(2)] The semigroup
$\{T_t: t\geq 0\}$ is AIUC if and only if there exists $R>0$ and $C_{V,R}>0$ such that
$$
\frac{V(x)}{\log |x|} \geq C_{V,R}.
$$
\end{enumerate}
\label{border}
\end{corollary}
\begin{proof}
Straightforward consequence of Theorems \ref{th:suff} and \ref{th:nec}, and \eqref{eq:Mq}.
\end{proof}
\noindent
The borderline case for fractional Schr\"odinger operators can be compared with the classic result for
the Feynman-Kac semigroup associated with Schr\"odinger operators $H = -\Delta+V$ which says that if
$V(x) = |x|^{\beta}$  the semigroup is IUC if and only if $\beta > 2$. Moreover, if $\beta > 2$, then
$c f(x) \le \varphi_0(x) \le C f(x)$, $|x| > 1$, holds with some $C,c>0$ and
$$
f(x) = |x|^{-\beta/4 + (d - 1)/2}\exp(-2 |x|^{1 + \beta/2}/(2+\beta)).
$$
For details see Cor. 4.5.5, Th. 4.5.11 and Cor. 4.5.8 in \cite{bib:D}, also \cite{bib:DS}.
\begin{remark}
\rm{
From the above it follows that for these processes IUC is a stronger property than AIUC. Indeed, consider
$$
V(x) = \log|x|\1_{\left\{|x|>1\right\}}(x) - \frac{1}{|x|^{\alpha/2}} \1_{\left\{|x|\leq1\right\}}(x).
$$
Then the Feynman-Kac semigroup $\{T_t: t\geq 0\}$ corresponding to $(-\Delta)^{\alpha/2} + V$ is AIUC
but it is not IUC. However, we do not know whether in the case of diffusions AIUC is a weaker property 
than IUC or not.
}
\label{classsicIUC}
\end{remark}

\begin{remark}
\rm{
From Corollary \ref{border} it follows that the condition on $V$ for the intrinsic ultracontractivity of the
semigroup generated by $(-\Delta)^{\alpha/2} + V$ is much weaker than in the case of $-\Delta + V$. This can
be explained by a pathwise interpretation of IUC. While this will be done elsewhere in detail,
we note that using the Feynman-Kac semigroup it is clear that the effect of the potential on the distribution
of paths is a concurrence of killing at a rate of $e^{-\int_0^t V_{+}(X_s)ds}$ and mass generation at a rate
of $e^{\int_0^t V_{-}(X_s)ds}$. However, if $V(x) \to \infty$ as $|x| \to \infty$, then outside some compact
set only the killing effect occurs and $\ex^x[e^{-\int_0^t V(X_s)ds}]$ gives the probability of survival of the
process under the potential up to time $t$. Asymptotically the probability of survival of the process staying
near the starting point $x$ is roughly $e^{-t V(x)}$, while the probability of surviving while travelling to
a region $D$ where the killing part of the potential is smaller is $\pr^x(X_t \in D)$. From \eqref{cond:probIU}
we see that $\{T_t: t\geq 0\}$ is IUC if and only if the probability that the process under $V$ survives up to
time $t$ far from $\inf V$ is bounded by the probability that the process survives up to time $t$ and is in
some bounded region $D$, independently of its starting point. Comparing these two probabilities suggests that
the outcome of the competing effects will be decided by the ratio $V(x)/|\log \pr^x(X_t \in D)|$. The following
examples support this interpretation. Take $D$ to be a bounded neighbourhood of the location of $\inf V$ (in the
examples below, the origin) and $x \in D^c$ such that $\dist(x,D)$ is large. Denote in each case below by $\pr^x$
the measure of the process with $V\equiv 0$.
\begin{enumerate}
\item[(1)]
\emph{Brownian motion:}
The expression $\pr^x(B_t \in D) = (4\pi t)^{-d/2} \int_D e^{\frac{|y-x|^2}{4t}}dy$ gives Gaussian tails
$$
C^{(1)}_t e^{-C^{(2)}_{t,D} |x|^2} \leq \pr^x(B_t \in D)\leq C^{(3)}_t e^{-C^{(4)}_{t,D} |x|^2},
$$
with $C^{(1)},..., C^{(4)} > 0$, leading to $-\log\pr^x(B_t \in D) \asymp |x|^2$ for the borderline case as
in \cite{bib:D}.
\item[(2)]
\emph{Symmetric stable process:} By using estimate \eqref{eq:weaksc} we derive that
$$
\pr^x(X_t \in D) 
\asymp t \frac{1}{|x|^{d+\alpha}}= t e^{-(d+\alpha) \log |x|}.
$$
This gives $-\log \pr^x(X_t \in D) \asymp \log|x|$ for the borderline case of the potential, which agrees
with Theorems \ref{th:suff} and \ref{th:nec}.

\item[(3)]
\emph{Relativistic stable process:}
Let $(X^m_t)_{t\geq 0}$ be a process in $\R^d$ with parameters $\alpha \in (0,2)$, $m>0$, generated by the
Schr\"odinger operator $(-\Delta + m^{2/\alpha})^{\alpha/2} - m + V$. It is proven in \cite{bib:KS} that in 
case of non-negative potentials comparable on unit balls the corresponding Schr\"odinger semigroup is IUC 
if and only if $\lim_{|x| \to \infty} \frac{V(x)}{|x|} = \infty$. Using estimates on the transition density
\cite{bib:Szt} we obtain
$$
C^{(1)} e^{-C^{(2)} |x|} \leq \pr^x(X^m_t \in D)\leq C^{(3)} e^{-C^{(4)} |x|},
$$
where $C^{(1)},...,C^{(4)}>0$ depend on $m,t$ and $D$ only, i.e., indeed $-\log \pr^x(X^m_t \in D)
\asymp |x|$.
\end{enumerate}
\label{whyeasier}
}
\end{remark}

\section{Gibbs measures for symmetric $\alpha$-stable processes}
\subsection{Existence of fractional $P(\phi)_1$-processes}
\noindent
In this section we prove that provided Assumption \ref{gs} holds, there exists a probability measure $\mu$
on $(D_{\rm r}(\R,\R^d),{\cB}(D_{\rm r}(\R,\R^d))$ such that for $f,g \in L^2(\Rd)$ and Kato-decomposable
potential $V$
\begin{equation}
(f, e^{-t \widetilde{H}_{\alpha}}g)= \ex_\mu \left[\overline{f(\widetilde X_0)}g(\widetilde X_t)\right],
\quad t \geq 0.
\end{equation}
We will identify the probability measure $\mu$ as the measure of the Markov process
$(\widetilde X_t)_{t \in \R}$ derived from the symmetric $\alpha$-stable process $(X_t)_{t \in \R}$ under $V$,
which we call \emph{fractional $P(\phi)_1$-process}. In the next subsection we show that, in fact, $\mu$ is a
Gibbs measure with respect to the stable bridge measure and potential $V$, and will analyze its uniqueness and
support properties.

For an interval or union of intervals $I \subset \R$ we denote by $\Omega_I = D_{\rm r}(I,\R^d)$ the space
of right continuous functions from $I$ to $\R^d$ with left limits, and by $\cF_I$ the $\sigma$-field generated
by the coordinate process $\omega(t)$, $\omega \in \Omega_I$, $t \in I$. Also, we will use the notations
$\Omega := \Omega_{\R}$, $\cF := \cF_{\R}$, and consider
a two-sided $\alpha$-stable process $\proo X$ with path space $\Omega$ as defined in Section 2.2.

\begin{theorem}
\label{th:exphi1}
Let $V$ be a Kato-decomposable potential, $(\widetilde{T}_t)_{t \geq 0}$ be the corresponding intrinsic
fractional Feynman-Kac semigroup. Denote by $\widetilde{X}_t(\omega) = \omega(t)$ the coordinate process on
$(\Omega, \cF)$ and consider the filtrations $\left(\cF_t^{+}\right)_{t \geq 0} = \sigma\left(\widetilde{X}_s: 0 \leq s \leq t\right)$, $\left(\cF_t^{-}\right)_{t \leq 0} = \sigma\left(\widetilde{X}_s: t \leq s \leq 0\right)$. Then for all $x \in \R^d$ there exists a probability measure $\mu^x$ on $(\Omega, \cF)$,
satisfying the properties below:
\begin{itemize}
\item[(1)]
$\mu^x\left(\widetilde{X}_t = x \right) = 1$.

\item[(2)]
\emph{Reflection symmetry:} $(\widetilde{X}_t)_{t \geq 0}$ and $(\widetilde{X}_t)_{t \leq 0}$ are
independent and
$$
\widetilde X_{-t} \stackrel{\rm d}{=} \widetilde X_t, \quad t \in \R.
$$
\item[(3)]
\emph{Markov property:}
$(\widetilde{X}_t)_{t \geq 0}$ is a Markov process with respect to $\left(\cF_t^{+}\right)_{t \geq 0}$,
and $(\widetilde{X}_t)_{t \leq 0}$ is a Markov processes with respect to $\left(\cF_t^{-}\right)_{t \leq 0}$.

\item[(4)]
\emph{Shift invariance:}
Let $-\infty < t_0 \leq t_1 \leq ... \leq t_n < \infty$. Then the finite dimensional distributions with respect to
the stationary distribution $\varphi_0^2 dx$ are given by
\begin{align}
\label{eq:fddist}
\int_{\R^d} \ex_{\mu^x}\left[\prod_{j=0}^n f_j(\widetilde{X}_{t_j})\right] \varphi^2_0(x) dx =
\left(f_0,\: \widetilde{T}_{t_1-t_0}\: f_1 ...\: \widetilde{T}_{t_n-t_{n-1}}\: f_n \right)_
{L^2(\R^d, \varphi_0^2 dx)}
\end{align}
for $f_j \in L^{\infty}(\R^d)$, $j=0,..., n$, and are shift invariant, i.e.,
$$
\int_{\R^d} \ex_{\mu^x}\left[\prod_{j=0}^n f_j(\widetilde{X}_{t_j})\right] \varphi^2_0(x) dx =
\int_{\R^d} \ex_{\mu^x}\left[\prod_{j=0}^n f_j(\widetilde{X}_{t_j+s})\right] \varphi^2_0(x) dx, \quad s \in \R.
$$
\end{itemize}
\end{theorem}

\medskip
We proceed now to prove Theorem \ref{th:exphi1} in several steps. Let $0 \leq t_0 \leq t_1 \leq ... \leq t_n$
and let the set function $\nu_{t_0, ..., t_n} : \times_{j=0}^n \cB(\R^d) \rightarrow \R$ be defined by
\begin{align}
\label{def:fdd1}
\nu_{t_0, ..., t_n}(\times_{j=0}^n A_j):= \left(\1_{A_0},\: \widetilde{T}_{t_1-t_0}\: \1_{A_1} ...\:
\widetilde{T}_{t_n-t_{n-1}}\: \1_{A_n} \right)_{L^2(\R^d, \varphi_0^2 dx)}
\end{align}
and
\begin{align}
\label{def:fdd2}
\nu_{t_0}(A) = \left(\1,\: \widetilde{T}_{t_0}\: \1_A \right)_{L^2(\R^d, \varphi_0^2 dx)} =
\left(\1,\: \1_A \right)_{L^2(\R^d, \varphi_0^2 dx)}.
\end{align}

\medskip
\noindent
\emph{(Step 1)} Denote $\cL = \left\{L \subset \R: \card(L) < \infty\right\}$. It can be verified 
directly that the family of set functions $(\nu_L)_{L \in \cL}$ given above satisfies the consistency 
condition
$$
\nu_{t_0, ..., t_{n+m}}\left((\times_{j=0}^n A_j) \times (\times_{j=n+1}^{n+m} \R^d)\right) =
\nu_{t_0, ..., t_n}(\times_{j=0}^n A_j).
$$
Hence by the Kolmogorov extension theorem there exists a probability measure $\nu_\infty$ on the space
$\left((\R^d)^{[0,\infty)}, \cM \right)$, where $\cM$ is the $\sigma$-field on $(\R^d)^{[0,\infty)}$ 
generated by all cylinder sets, such that
\begin{align*}
& \nu_t(A) = \ex_{\nu_\infty}[\1_A(Y_t)], \\
& \nu_{t_0,...,t_n}\left( \times_{j=0}^n A_j \right) = \ex_{\nu_\infty}\left[\prod_{j=0}^n \1_{A_j}(Y_{t_j})\right],
\quad n \geq 1 \ ,
\end{align*}
where $Y_t(\omega) = \omega(t)$, $\omega \in (\R^d)^{[0,\infty)}$, is the coordinate process. Thus the stochastic
process $(Y_t)_{t\geq0}$ on the probability space $\left((\R^d)^{[0,\infty)}, \cM, \nu_\infty \right)$
satisfies
\begin{align}
\label{eq:fdd1}
& \ex_{\nu_\infty}\left[\prod_{j=0}^n f_j(Y_{t_j})\right] = \left(f_0,\: \widetilde{T}_{t_1-t_0}\: f_1 ...\:
\widetilde{T}_{t_n-t_{n-1}}\: f_n \right)_{L^2(\R^d, \varphi_0^2 dx)} \\
\label{eq:fdd2}
& \ex_{\nu_\infty}[f_0(Y_{t_0})] = \left(\1,\: \widetilde{T}_{t_0}\: f_0 \right)_{L^2(\R^d, \varphi_0^2 dx)} =
\left(\1,\: f_0 \right)_{L^2(\R^d, \varphi_0^2 dx)}
\end{align}
for $f_j \in L^{\infty}(\R^d)$, $j=0,1,...,n$, $0 \leq t_0 \leq t_1 \leq ... \leq t_n$.
Notice that the right hand side of \eqref{eq:fdd1} can be expressed directly in terms of symmetric $\alpha$-stable
process $(X_t,\pr^x)_{t \geq 0}$, i.e.,
\begin{align}
\label{eq:fdddir}
\ex_{\nu_\infty}\left[\prod_{j=0}^n f_j(Y_{t_n})\right] =
\int_{\R^d} \varphi_0(x) \ex^x\left[e^{-\int_0^{t_n} (V(X_s)-\lambda_0) ds}
\prod_{j=0}^n f_j(X_{t_j})\varphi_0(X_{t_n})\right]dx.
\end{align}

\medskip
\noindent
\emph{(Step 2)} Next we prove the existence of a c\`adl\`ag and a c\`agl\`ad version of the above process.
In this step we check the standard Dynkin-Kinney type condition \cite[p. 59-62]{bib:Sat} for this process. Let $M
\subset [0,\infty)$ and $\varepsilon >0$ and fix $\omega$. The function $Y_t(\omega)$ is said to have
$\varepsilon$-oscillation $n$ times in $M$ if there are $t_0 < t_1 < ... < t_n$ in $M$ such that $|Y_{t_j}(\omega)
- Y_{t_{j-1}}(\omega)|>\varepsilon$ for $j=1,...,n$. Also, $Y_t(\omega)$ has $\varepsilon$-oscillation infinitely
often in $M$ if for every $n$, $Y_t(\omega)$ has $\varepsilon$-oscillation $n$ times in $M$. Let
\begin{align*}
& \Omega_2 = \left\{\omega: \lim_{s \in \qpr, s \downarrow t} Y_s(\omega) \ \text{exists in} \ \R^d \
\text{for all} \ t\geq0 \ \text{and} \ \lim_{s \in \qpr, s \uparrow t} Y_s(\omega) \ \text{exists in} \ \R^d
\ \text{for all} \ t>0 \right\}, \\
& A_{N, k} = \left\{\omega: Y_t(\omega) \ \text{does not have} \ \frac{1}{k} \text{-oscillation infinitely often in}
\ [0,N] \cap \qpr \right\}, \\
& \Omega_2^{'} = \bigcap_{N=1}^{\infty} \bigcap_{k=1}^{\infty} A_{N,k}.
\end{align*}
Clearly, $\Omega_2^{'} \in \cF$. Moreover, it is proven in \cite[Lemma 11.2]{bib:Sat} that $\Omega_2^{'} \subset
\Omega_2$. Define
$$
B(p,\varepsilon,M) = \left\{\omega: Y_t(\omega) \ \text{has} \ \varepsilon \text{-oscillation} \ p \ \text{times in}
\ M\right\}.
$$

\begin{lemma}
\label{lm:sc}
The following assertions follow:
\begin{itemize}
\item[(1)] For every $\varepsilon > 0$ we have
$$
\nu_\infty\left(\left\{\omega: |Y_t(\omega) - Y_s(\omega)| > \varepsilon \right\}\right) \to 0
\quad  \text{as} \quad  |t-s| \to 0.
$$
\item[(2)] $\nu_\infty\left(\Omega^{'}_2\right) = 1.$
\end{itemize}
\end{lemma}

\begin{proof}
To show (1) let $0 \leq s < t$. By \eqref{eq:fdddir}
\begin{align*}
\nu_\infty& \left(\left\{\omega: |Y_t(\omega) - Y_s(\omega)| > \varepsilon \right\}\right)  =
\int_{\R^d} \varphi_0(x) \ex^x\left[e^{-\int_0^{t-s} (V(X_r) - \lambda_0) dr}  \varphi_0(X_{t-s})
\1_{B^c(x, \varepsilon)}(X_{t-s})\right] dx \\
& \leq
\int_{\R^d} \varphi_0(x) dx \sup_{x \in \R^d} \left(\ex^x\left[e_{2(V-\lambda_0)}(t-s) \varphi_0^2(X_{t-s})\right]\right)^{1/2}
\left(\pr^x(X_{t-s} \in B^c(x,\varepsilon))\right)^{1/2}\\
& \leq
\left\|\varphi_0\right\|_1 \left\|\varphi_0^2\right\|_{\infty} \left(e^{C^{(1)}_V + C^{(2)}_V(t-s)}\right)^{1/2}
\left(\pr^0(X_{t-s} \in B^c(0,\varepsilon))\right)^{1/2},
\end{align*}
which goes to $0$ as $|t-s| \to 0$ by stochastic continuity of the symmetric stable process $\pro X$.

To prove (2) observe that it suffices to show that $\nu_\infty(A_{N,k}^c) = 0$ for any fixed $N$ and $k$. Again,
using stochastic continuity of $\pro X$, choose $l$ large enough so that
$$
\pr^0\left(X_{N/l} \in B^c\left(0,\frac{1}{4k}\right)\right) < 1/2.
$$
We have
\begin{align*}
\nu_\infty(A_{N,k}^c) & = \nu_\infty\left(\left\{\omega: Y_t(\omega) \ \text{has $\frac{1}{k}$-oscillation infinitely
often in} \ [0,N]\cap \qpr \right\}\right) \\
& \leq\sum_{j=1}^l \nu_\infty\left(\left\{\omega: Y_t(\omega) \ \text{has $\frac{1}{k}$-oscillation infinitely often
in} \ \left[\frac{j-1}{l}N,\frac{j}{l}N \right]\cap \qpr \right\}\right) \\
& = \sum_{j=1}^l \lim_{p \to \infty} \nu_\infty\left(B\left(p, \frac{1}{k}, \left[\frac{j-1}{l}N,\frac{j}{l}N \right]
\cap \qpr \right)\right).
\end{align*}
Enumerating the elements of $\left[\frac{j-1}{l}N,\frac{j}{l}N \right] \cap \qpr$ as $t_1, t_2, ...$, we have
$$
\nu_\infty \left(B\left(p, \frac{1}{k}, \left[\frac{j-1}{l}N,\frac{j}{l}N \right] \cap \qpr \right)\right) =
\lim_{n \to \infty} \nu_\infty\left(B\left(p, \frac{1}{k}, \left\{t_1, ..., t_n \right\} \right)\right).
$$
Moreover, by \eqref{eq:fdddir} we get
\begin{align*}
\nu_\infty & \left(B\left(p, \frac{1}{k}, \left\{t_1, ..., t_n \right\} \right)\right) \\
& = \int_{\R^d} \varphi_0(x) \ex^x\left[e^{-\int_0^{N/l} (V(X_s) - \lambda_0) ds}
\1_{B\left(p, \frac{1}{k}, \left\{t_1, ..., t_n \right\} \right)} \varphi_0\left(X_{N/l}\right)\right] dx\\
& \leq \int_{\R^d} \varphi_0(x) dx \sup_{x \in \R^d} \left(\ex^x\left[e_{2(V-\lambda_0)}
\left(N/l\right) \varphi_0^2(X_{N/l})\right]\right)^{1/2} \left(\pr^x\left(B\left(p, \frac{1}{k},
\left\{t_1, ..., t_n \right\} \right)\right)\right)^{1/2}\\
& \leq \left\|\varphi_0\right\|_1 \left\|\varphi_0^2\right\|_{\infty}
\left(e^{C^{(1)}_V + C^{(2)}_V(N/l)}\right)^{1/2} \sup_{x \in \R^d} \left(\pr^x\left(B\left(p, \frac{1}{k},
\left\{t_1, ..., t_n \right\} \right)\right)\right)^{1/2}.
\end{align*}
Since by \cite[Lm. 11.4]{bib:Sat}
$$
\pr^x\left(B\left(p, \frac{1}{k}, \left\{t_1, ..., t_n \right\} \right)\right) \leq \left(2
\pr^0\left(X_{N/l} \in B^c\left(0,\frac{1}{4k}\right)\right) \right)^p,
$$
we have $\nu_\infty(A_{N,k}^c) = 0$ and the lemma is proved.
\end{proof}

\begin{lemma}
The process $(Y_t)_{t \geq 0}$ has a right continuous version with left limits (i.e., c\`adl\`ag) and a left
continuous version with right limits (i.e., c\`agl\`ad) with respect to the measure $\nu_\infty$.
\end{lemma}

\begin{proof}
The existence of a c\`adl\`ag version is a consequence of Lemma \ref{lm:sc} and the standard arguments in the
proof of \cite[Lm. 11.3]{bib:Sat}. In the same way we show the existence of a c\`agl\`ad version of the process
$(Y_t)_{t \geq 0}$.
\end{proof}
\noindent
Let now $(Y^{'}_t)_{t \geq 0}$ be the c\`adl\`ag version of $(Y_t)_{t \geq 0}$ on
$\left((\R^d)^{[0,\infty)}, \cM, \nu_\infty \right)$. Recall that $\Omega_{[0,\infty)} =
D_{\rm r}(\R^{+}, \R^d)$. Denote the image measure of $\nu_\infty$ on $(\Omega_{[0,\infty)}, \cF_{[0,\infty)})$
by
$$
\cQ  = \nu_\infty \circ (Y^{'}_t)^{-1}.
$$
We identify the coordinate process by $\widetilde{Y}_t(\omega) = \omega(t)$, for $\omega \in \Omega_{[0,\infty)}$.
Thus we have constructed a random process $(\widetilde{Y}_t)_{t \geq 0}$ on
$(\Omega_{[0,\infty)}, \cF_{[0,\infty)}, \cQ)$ such that $Y'_t \stackrel{\rm d}{=} \widetilde{Y_t}$. Then
\eqref{eq:fdd1} and \eqref{eq:fdd2} can be expressed in terms of
$\pro {\widetilde{Y}}$ as
\begin{align*}
& \left(f_0,\: \widetilde{T}_{t_1-t_0}\: f_1 ...\: \widetilde{T}_{t_n-t_{n-1}}\: f_n \right)_
{L^2(\R^d, \varphi_0^2 dx)} = \ex_{\cQ}\left[\prod_{j=0}^n f_j(\widetilde{Y}_{t_n})\right] ,\\
& \left(\1,\: \widetilde{T}_{t_0}\: f_0 \right)_{L^2(\R^d, \varphi_0^2 dx)} =
\left(\1,\: f_0 \right)_{L^2(\R^d, \varphi_0^2 dx)} = \ex_{\cQ}[f_0(\widetilde{Y}_t)].
\end{align*}
Note that by considering the c\`agl\`ad version of the process $(Y_t)_{t \geq 0}$, we can also construct a
random process on the space of the left continuous functions with left limits $D_l(\R^{+},\R^d)$ satisfying
the above equalities.

\medskip
\noindent
\emph{(Step 3)} Define a family of measures on $\left(\Omega_{[0,\infty)}, \cF_{[0,\infty)} \right)$ by
$$
\cQ^x (\;\cdot\;) = \cQ(\;\cdot\;|\widetilde{Y}_0=x), \quad x \in \R^d.
$$
Since the distribution of $Y_0$ is $\varphi^2_0(x)dx$, we have $\cQ(A)=\int_{\R^d} \varphi_0^2(x)
\ex_{\cQ^x}[\1_A] dx$. Then the process $(\widetilde{Y}_t)_{t \geq 0}$ on
$\left(\Omega_{[0,\infty)}, \cF_{[0,\infty)}, \cQ^x \right)$ satisfies
\begin{align}
\label{eq:fdd3}
& \left(f_0,\: \widetilde{T}_{t_1-t_0}\: f_1 ...\: \widetilde{T}_{t_n-t_{n-1}}\: f_n \right)_
{L^2(\R^d, \varphi_0^2 dx)}
= \int_{\R^d} \varphi_0^2(x)\ex_{\cQ^x}\left[\prod_{j=0}^n f_j(\widetilde{Y}_{t_j})\right] dx,\\
\label{eq:fdd4}
& \left(\1,\: \widetilde{T}_{t_0}\: f_0 \right)_{L^2(\R^d, \varphi_0^2 dx)} =
\left(\1,\: f_0 \right)_{L^2(\R^d, \varphi_0^2 dx)} =
\int_{\R^d} \varphi_0^2(x)  \ex_{\cQ^x}[f_0(\widetilde{Y}_t)] dx.
\end{align}

\begin{lemma}
$(\widetilde{Y}_t)_{t \geq 0}$ is a Markov process on
$\left(\Omega_{[0,\infty)}, \cF_{[0,\infty)}, \cQ^x \right)$ with respect to the natural filtration
$(\cG_t)_{t \geq 0}$, where $\cG_t = \sigma\left(\widetilde{Y}_s, 0 \leq s \leq t \right)$.
\end{lemma}

\begin{proof}
Let
$$
\widetilde{u}_t(x,A) = \widetilde{T}_t \1_A(x),
$$
for every $A \in \cB(\R^d)$, $x \in \R^d$ and $t\geq 0$. Clearly, $\widetilde{u}_t(x,A) =
\ex_{\cQ^x}[\1_A(\widetilde{Y}_t)]$ and, by \eqref{eq:fdd3} and \eqref{eq:fdd4}, the finite dimensional
distributions of $(\widetilde{Y}_t)_{t \geq 0}$ are given by
\begin{align}
\label{eq:fdd5}
\ex_{\cQ^x}\left[\prod_{j=0}^n \1_{A_j}(\widetilde{Y}_{t_j})\right] = \int \prod_{j=0}^n \1_{A_j}(x_j)
\prod_{j=0}^n \widetilde{u}_{t_j - t_{j-1}}(x_{j-1},dx_j), \quad t_0 = 0, \ \ x_0 = x.
\end{align}
By using the properties of the intrinsic fractional semigroup $(\widetilde{T}_t)_{t \leq 0}$ it can be
checked directly that $\widetilde{u}_t(x,A)$ is a probability transition kernel, thus
$(\widetilde{Y}_t)_{t \geq 0}$ is a Markov process with finite dimensional distributions given by
\eqref{eq:fdd5}.
\end{proof}

\medskip
\noindent
\emph{(Step 4)} We now extend $(\widetilde{Y}_t)_{t \geq 0}$ to a process on the whole real line $\R$. Consider
$\widehat{\Omega}= D_r(\R^{+}, \R^d) \times D_l(\R^{+}, \R^d)$ with an appropriate product $\sigma$-field 
$\widehat{\cF}$ and product measure $\widehat{\cQ^x}$, respectively. Let $\widehat{X}_t$ be the coordinate 
process given by
$$
\widehat X_t(\omega)=
\left\{
\begin{array}{ll}\omega_1(t), & t\geq0\\
\omega_2(-t),                 & t<0
\end{array} \right.
$$
for $\omega =(\omega_1, \omega_2) \in \widehat{\Omega}$. We thus defined a stochastic process
$(\widehat{X}_t)_{t \in \R}$ on the product space $(\widehat{\Omega}, \widehat{\cF}, \widehat{\cQ^x})$ such
that $\widehat{\cQ^x}(\widehat{X}_0 = x)=1$ and $\R \ni t \mapsto \widehat{X}_t(\omega)$ is right continuous
with left limits. It is easy to see that $\widehat{X}_t$, $t \geq 0$, and $\widehat{X}_s$, $s \leq 0$, are
independent, and $\widehat{X}_t \stackrel{\rm d}{=}\widehat{X}_{-t}$.

\medskip
\noindent
\emph{(Step 5)} We now prove Theorem \ref{th:exphi1}.
\begin{proof}[Proof of Theorem \ref{th:exphi1}]
Recall that $\Omega = D_{\rm r}(\R,\R^d)$. Denote the image measure of $\widehat{\cQ}^x$ on $(\Omega, \cF)$
with respect to $\widehat{X}$ by
$$
\mu^x = \widehat{\cQ}^x \circ \widehat{X}^{-1}.
$$
Let $\widetilde{X}_t(\omega) = \omega(t)$, $t \in \R$, $\omega \in \Omega$, denote the coordinate process.
Clearly, we have
$$
\widetilde{X}_t \stackrel{\rm d}{=} \widetilde{Y}_t, \;\; t \geq 0, \quad \mbox{and} \quad
\widetilde{X}_t \stackrel{\rm d}{=} \widetilde{Y}_{-t}, \; \; t \leq 0.
$$
Thus we see that $\widetilde{X}_t \stackrel{\rm d}{=} \widetilde{X}_{-t}$ and by Step 4, $(\widetilde{X}_t)_
{t \geq 0}$ and $(\widetilde{X}_t)_{t \leq 0}$ are independent. Furthermore, by Step 2, $(\widetilde{Y}_t)_
{t \geq 0}$ and $(\widetilde{Y}_{-t})_{t \leq 0}$ are Markov processes respectively under the natural
filtrations $\sigma(\widetilde{Y}_s, 0 \leq s \leq t)$ and $\sigma(\widetilde{Y}_s, 0 \leq s \leq -t)$. Thus
$(\widetilde{X}_t)_{t \geq 0}$ and $(\widetilde{X}_{t})_{t \leq 0}$ are also Markov processes with respect to
$(\cF^{+}_t)_{t \geq 0}$ and $(\cF^{-}_t)_{t \leq 0}$.

It remains to show assertion (4) of the theorem. Let $t_0 \leq t_1 \leq ... \leq t_n \leq 0 \leq t_{n+1} \leq ...
\leq t_{n+m}$ and $f_j \in L^{\infty}(\R^d)$ for $j=0,1,...,n+m$. By independence of $(\widetilde{X}_t)_{t \geq 0}$
and $(\widetilde{X}_{t})_{t \leq 0}$ we have
$$
\int_{\R^d} \ex_{\mu^x}\left[\prod_{j=0}^{n+m} f_j(\widetilde{X}_{t_j})\right] \varphi^2_0(x) dx =
\int_{\R^d} \ex_{\mu^x}\left[\prod_{j=0}^n f_j(\widetilde{X}_{t_j})\right] \ex_{\mu^x}\left[\prod_{j=n+1}^{n+m}
f_j(\widetilde{X}_{t_j})\right] \varphi^2_0(x) dx.
$$
Moreover,
$$
\ex_{\mu^x}\left[\prod_{j=n+1}^{n+m} f_j(\widetilde{X}_{t_j})\right] =
\left(\widetilde{T}_{t_{n+1}}\: f_{n+1}\widetilde{T}_{t_{n+2}-t_{n+1}}\: f_{n+2} ...\:
\widetilde{T}_{t_{n+m}-t_{n+m-1}}\: f_{n+m} \right)(x)
$$
and
$$
\ex_{\mu^x}\left[\prod_{j=0}^{n} f_j(\widetilde{X}_{t_j})\right] =
\ex_{\mu^x}\left[\prod_{j=0}^{n} f_j(\widetilde{X}_{-t_j})\right] =
\left(\widetilde{T}_{-t_n}\: f_n \widetilde{T}_{t_n-t_{n-1}}\: f_{n-1} ...\: \widetilde{T}_{t_1-t_0}
\: f_0 \right)(x).
$$
Hence
\begin{align*}
& \int_{\R^d} \ex_{\mu^x}\left[\prod_{j=0}^{n+m} f_j(\widetilde{X}_{t_j})\right] \varphi^2_0(x) dx \\
& =  \left(\widetilde{T}_{-t_n} f_n \widetilde{T}_{t_n-t_{n-1}} f_{n-1} ...\:
\widetilde{T}_{t_1-t_0} f_0, \widetilde{T}_{t_{n+1}} f_{n+1}\widetilde{T}_{t_{n+2}-t_{n+1}} f_{n+2} ...\:
\widetilde{T}_{t_{n+m}-t_{n+m-1}} f_{n+m} \right)_{L^2(\R^d, \varphi_0^2 dx)}\\
& = \left(f_0, \widetilde{T}_{t_1-t_0} f_1 ...\: \widetilde{T}_{t_{n+m}-t_{n+m-1}} f_{n+m} \right)_
{L^2(\R^d, \varphi_0^2 dx)}
\end{align*}and \eqref{eq:fddist} follows. Shift invariance is a simple consequence of the above equality.
\end{proof}

\begin{definition}[\textbf{Fractional $P(\phi)_1$-process}]
\rm{
We call the process $(\widetilde{X}_t, \mu^x)_{t \in \R}$ obtained in Theorem \ref{th:exphi1} the
\emph{fractional $P(\phi)_1$-process} for the Kato-decomposable potential $V$. We call the measure
$\mu$ on $(\Omega, \cF)$ with
$$
\mu(A)= \int_{\R^d} \ex_{\mu^x}\left[\1_A \right] \varphi^2_0(x) dx
$$
\emph{fractional $P(\phi)_1$-measure} for the Kato-decomposable potential $V$.
}
\end{definition}
\noindent
For our purposes below it will be useful to see $\mu$ as the measure with respect to the stable bridge.

\begin{lemma} We have for $A \in \cF_{[s,t]}$, $s, t \in \R$,
\begin{align}
\label{eq:useexp}
\mu(A) = \int_{\R^d}dx  \varphi_0(x) \int_{\R^d} dy  \varphi_0(y)
\int_{\Omega} e^{-\int_s^t (V(X_r(\omega))-\lambda_0) dr} \1_A d\nu_{[s,t]}^{x,y}(\omega).
\end{align}
\end{lemma}

\begin{proof}
It is enough to check that the equality \eqref{eq:useexp} holds for cylinder sets of the form
$A = \left\{\omega(t_0) \in B_0, ..., \omega(t_n) \in B_n \right\}$, where $s \leq t_0 < t_1
< ... < t_n < t$ and $B_1, B_2, ..., B_n$ are Borel sets. This can be seen directly by
\eqref{eq:fddist}, the Markov property of the symmetric stable process $(X_t)_{t \geq 0}$, the
fact that $(X_t, \pr^{s,x}) \stackrel{\rm d}{=} (X_{t-s},\pr^x)$, and the equalities \eqref{eq:cnd},
\eqref{eq:xmeas}.
\end{proof}

\subsection{Properties of fractional $P(\phi)_1$-processes}
\noindent
In this subsection we show that the behaviour of Kato-decomposable potentials $V$ at infinity
(in particular, AIUC semigroups) has a direct influence on the properties of $P(\phi)_1$-processes.
A consequence of the construction in the previous subsection is that a $P(\phi)_1$-process is
a stationary Markov process with stationary distribution $\rho(A)=\int_A \varphi_0^2(y) dy$,
i.e., $\mu(\widetilde X_t \in A) = \rho(A)$ for every $t \in \R$ and Borel set $A$.

\begin{theorem}
Let $V$ be a Kato-decomposable potential, and consider the following properties:
\begin{itemize}
\item[(1)]
The semigroup $(T_t)_{t \geq 0}$ is AIUC.
\item[(2)]
There exists $t_0 > 0$ such that for every $t \geq t_0$ we have
$$
\sup_{x \in \R^d} \ex_{\mu^x}\left[\varphi_0^{-1}(\widetilde X_t)\right] < \infty.
$$
\item[(3)]
For every Borel set $A \in \R^d$
$$
\lim_{t \to \infty} \mu^x (\widetilde X_t \in A) = \rho(A)
$$
holds, uniformly in $x \in \R^d$.
\end{itemize}
Then we have $(1) \Longleftrightarrow (2) \Longrightarrow (3)$.
\end{theorem}
\begin{proof}
The implication $(1) \Longrightarrow (3)$ is a direct consequence of Lemma \ref{lm:iucons}. To
prove equivalence of (1) and (2) it suffices to see that AIUC is equivalent to the property
that there exists $t_0 > 0$ such that for every $t \geq t_0$ there exists a constant $C_{V,t}$
such that for every $x \in \R^d$ we have $T_t \1 (x) \leq C_{V,t} \varphi_0(x)$. However, this
is trivially equivalent to (2).
\end{proof}

The asymptotic behaviour of the ground state allows to estimate the actual support of $\mu$.

\begin{theorem}[\textbf{Typical path behaviour}]
\label{th:subspaceom}
Let $V$ be Kato-decomposable and $\varphi_0 \in L^2(\R^d) \cap L^1(\R^d)$. Also, let $(a_n)_{n \geq 1}$
be a sequence of positive real numbers such that $\sum_{n=1}^{\infty} a_n < \infty$. Then
\begin{align}
\label{eq:subspaceom}
\lim_{|N| \rightarrow \infty}
\frac{a_{|N|}}{\varphi_0(\omega(N))} & = 0, \quad \mbox{$\mu$-a.s.}
\end{align}
\end{theorem}

\begin{proof}
By time reversibility of $\mu$ it suffices to show that for every $\varepsilon > 0$
\begin{align}
\label{eq:BCL}
\mu\left(\limsup_{N \rightarrow \infty} \frac{a_N}{\varphi_0(\omega(N))} > \varepsilon \right)  = 0.
\end{align}
The fact that $\varphi_0 \in L^1(\R^d)$ and stationarity of $\mu$ give
\begin{align*}
\mu\left(\frac{a_N}{\varphi_0(\omega(N))} > \varepsilon \right) =
\mu\left(\frac{a_N}{\varepsilon} > \varphi_0(\omega(0)) \right)
= \int_{\R^d} \1_{\left\{\varphi_0< a_N/\varepsilon \right\}}(x) \varphi_0^2(x)dx
\leq \frac{a_N}{\varepsilon} \left\|\varphi_0\right\|_1 .
\end{align*}
Since the right hand side of the above inequality is summable with respect to $N$ for every $\varepsilon >0$,
the Borel-Cantelli Lemma gives \eqref{eq:BCL} for every $\varepsilon > 0$, and \eqref{eq:subspaceom} follows.
\end{proof}

\begin{corollary}
Under the assumptions of the above theorem, by taking $a_n =  n^{-1-\theta}$, $\theta>0$, we obtain that the
measure $\mu$ is supported by a subset of paths such that for every $\theta>0$
\begin{align}
\label{eq:theta}
\lim_{|N| \rightarrow \infty} \frac{1}{|N|^{1+\theta}\varphi_0(\omega(N))} = 0.
\end{align}
\end{corollary}

By using Theorem \ref{th:eig1}, a more explicit
description of the support for a wide class of potentials can be given.

\begin{corollary}
\label{cor:subspace1}
Let $V$ be Kato-decomposable such that $V(x) \to \infty$ as $|x| \to \infty$. Assume that there exists
a compact set $K \in \R^d$, possibly empty, such that
\begin{itemize}
\item[(1)]
$K^c \subset \left\{x \in \Rd: \overline{B}(x,1) \cap \supp(V_{-}) = \emptyset \right\}$,
\item[(2)]
there is a constant $M_{V,K}\geq1$ such that for every $x \in K^c$
\begin{align*}
V_{+}(y) & \le M_{V,K} V_{+}(z), \quad y, z \in B(x,1).
\end{align*}
\end{itemize}
Then for every $\theta>0$ we have
\begin{align*}
\lim_{|N| \rightarrow \infty} \frac{V_{+}(\omega(N))|\omega_N|^{d+\alpha}
\1_{K^c}(\omega(N))}{|N|^{1+\theta}} = 0,  \quad \mbox{$\mu$-a.s.}
\end{align*}
\end{corollary}
\begin{proof}
By Corollary \ref{th:eig1} we have that $\varphi_0(x)$ and $(V_{+}(x)|x|^{d+\alpha})^{-1}$ are
comparable on $K^c$. Since $0<C_1 \leq \varphi_0 \leq C_2 < \infty$ on $K$, the assertion
follows from the previous theorem.
\end{proof}
\noindent
Some examples illustrating the above typical path behaviour results are discussed below.

\subsection{Existence of Gibbs measures}
\noindent
In this section we show that the measure of a $P(\phi)_1$-process for a potential $V$ is a Gibbs
measure for the same potential.

Without restricting generality we consider symmetric intervals $I = [-T,T]$, $T>0$. We will use
the notations $\cF_T := \cF_{[-T,T]}$, $\cT_T := \cF_{(-\infty,-T] \cup [T, \infty)}$, $\nu^{x,y}_T
= \nu^{x,y}_{[-T,T]}$. Let $\bar\omega \in \Omega$, and consider the point measure $\delta^{\bar{\omega}}_T$
on $\Omega_{[-T,T]^c}$ concentrated on $\bar{\omega} \in \Omega$. For every $T>0$ we define a measure
on $(\Omega,\cF)$ by
\begin{align}
\label{eq:refmeas}
\nu^{\bar{\omega}}_T := \nu^{\bar{\omega}(-T),\bar{\omega}(T)}_T \otimes \delta^{\bar{\omega}}_T
\, .
\end{align}
In what follows we consider the family of measures $(\nu^{\bar{\omega}}_T)_{T>0}$ as reference measure.

Let $V$ be a Kato-decomposable potential and define
\begin{align}
\label{eq:integr}
Z_T(x,y) := \int_\Omega e^{-\int_{-T}^T V(X_s(\omega)) ds}d\nu^{x,y}_T(\omega)
\, ,
\end{align}
for all $T>0$ and all $x,y \in \R^d$. By Lemma \ref{lm:kernel} (5) we have
$$
Z_T(x,y) = u(2T,x,y) < \infty, \quad  x, y \in \R^d, \;T>0.
$$
For every $T>0$ define the conditional probability kernel
\begin{align}
\label{def:probkern}
\mu_T(A,\bar{\omega}) & = \frac{1}{Z_T(\bar{\omega}(-T),\bar{\omega}(T))} \int_\Omega \1_A(\omega)
e^{-\int_{-T}^T V(X_s(\omega)) ds}d\nu^{\bar{\omega}}_T(\omega), \quad A \in \cF, \; \bar{\omega}
\in \Omega.
\end{align}
We refer to $\bar\omega$ as a boundary path configuration.

\begin{definition}[\textbf{Gibbs measure}]
\rm{
A probability measure $\mu$ on $(\Omega,\cF)$ is called a \emph{Gibbs measure} for the fractional
$P(\phi)_1$-process $(\widetilde X_t)_{t \in \R}$ with potential $V$ if for every $A \in \cF$ and
every $T >0$ the function $\bar{\omega} \mapsto \mu_T(A,\bar{\omega})$ is a version of the conditional
probability $\mu(A | \cT_T)$, i.e.,
\begin{equation}
\label{DLR}
\mu(A | \cT_T)(\bar\omega) = \mu_T(A,\bar{\omega}), \quad A \in \cF, \; T > 0, \; \mbox{a.e.} \;
\bar\omega \in \Omega.
\end{equation}
\label{def:GibbsMeasure}
}
\end{definition}
\noindent
Condition (\ref{DLR}) is traditionally called \emph{Dobrushin-Lanford-Ruelle (DLR) equations}.

\begin{theorem}
\label{exist}
Let $\mu$ be the $P(\phi)_1$-measure for the Kato decomposable-potential $V$. For every $T>0$, $\bar{\omega}
\in \Omega$ and $A \in \cF$, $\bar{\omega} \mapsto \mu_T(A,\bar{\omega})$ is a version of the conditional
probability $\mu(A | \cT_T)(\bar{\omega})$, hence $\mu$ is a Gibbs measure for $V$.
\end{theorem}

\begin{proof}
Let $0<S<T$, $A \in \cF_S$, $B_1 \in \cF_{[-T,-S]}$, $B_2 \in \cF_{[S,T]}$, $B = B_1 \cap B_2 \in
\cF_{[-T,-S] \cup [S,T]}$. By a monotone class argument, it suffices to consider sets of
the form $A \cap B$. In order to show $\mu(\mu_S(A \cap B, \cdot)) = \mu(A \cap B)$ first note that since
$\nu_T^{\xi,\eta}(\{\bar\omega(-T) \neq \xi\}) = \nu_T^{\xi,\eta}(\{\bar\omega(T) \neq \eta\}) = 0$, we have
$$
\int_\Omega e^{-\int_{-S}^S V(X_s(\bar\omega)) \, ds} \mu_S(A,\bar\omega) \, d\nu_S^{\xi,\eta}(\bar\omega) =
\int_\Omega e^{-\int_{-S}^S V(X_s(\bar\omega)) \, ds} \1_A(\bar\omega) \, d\nu_S^{\xi,\eta}(\bar\omega).
$$
Then the Markov property of $(X_t)_{t \in \R}$ yields
\begin{eqnarray*}
\lefteqn{ \int_\Omega e^{-\int_{-T}^T V(X_s(\bar\omega)) \, ds} \mu_S(A \cap B,\bar\omega) \,
d\nu_T^{x,y}(\bar\omega) } \\
 & = &
 \int_{\R^d\times \R^d}
 \left( \int_\Omega e^{-\int_{-T}^{-S} V(X_s(\bar\omega)) \, ds} \1_{B_1}(\bar\omega) \, d\nu_{[-T,-S]}^{x,\xi}(\bar\omega)
        \right) \left( \int_\Omega e^{-\int_{-S}^S V(X_s(\bar\omega)) \, ds} \mu_S(A,\bar\omega) \, d\nu_S^{\xi,\eta}(\bar\omega)
        \right)\\
 &&  \times \left( \int_\Omega e^{-\int_S^T V(X_s(\bar\omega)) \, ds} \1_{B_2}(\bar\omega) \, d\nu_{[S,T]}^{\eta,y}(\bar\omega)
     \right) \, d\xi d\eta \\
 & = &
 \int_\Omega e^{-\int_{-T}^T V(X_s(\bar\omega)) \, ds} \1_{A \cap B}(\bar\omega) \, d\nu_T^{x,y}(\bar\omega)
\end{eqnarray*}
for all $x,y \in \R^d$. By \eqref{eq:useexp}, we plainly obtain
\begin{equation}
\int_\Omega \mu_S (A \cap B, \bar{\omega}) \, d\mu(\bar{\omega}) = \mu(A \cap B).
\end{equation}
As $\bar{\omega} \mapsto \mu_S(C,\bar{\omega})$ is $\cT_S$-measurable, the proposition is proven.
\end{proof}

\subsection{Uniqueness and support properties}

It is seen above that a $P(\phi)_1$-measure is a Gibbs measure for the given potential $V$. In fact, the
existence of a Gibbs measure $\mu$ follows from the existence of the ground state $\varphi_0$ of the
operator $(-\Delta)^{\alpha/2} + V$. However, it is not clear whether there are any other probability
measures on $(\Omega, \cF)$ satisfying the DLR equations for the potential $V$. This problem will be
discussed in this section.

In the case of the Schr\"odinger operator $(-1/2)\Delta + V$ the case of one-dimensional Ornstein-Uhlenbeck 
process obtained for $V(x) = \frac{1}{2}(x^2 - 1)$ shows that uniqueness need not hold in general (see 
\cite[Ex. 3.1]{bib:BL}). In fact, in this case there are uncountably many Gibbs measures supported on 
$C(\R,\R)$ for this potential.

We start with two lemmas concerning uniqueness, which were proved in \cite{bib:BL} in the case of Gibbs measures
on Brownian motion. The first lemma gives a simple criterion allowing to check if a Gibbs measure is the only
one supported on a given set. Its proof uses the same arguments as the classical one and we omit it. Recall that a
probability measure $P$ is said to be supported on a set $B$ if $P(B) = 1$.

\begin{lemma}
\label{lm:uniq}
Let $\Omega^{*} \subset \Omega$ be measurable and $\nu$ be a Gibbs measure for the potential $V$ such that
$\nu(\Omega^{*})=1$. Suppose that for every $T >0$, $B \in \cF_T$ and $\bar{\omega} \in \Omega^{*}$,
$\nu_N(B,\bar{\omega}) \rightarrow \nu(B)$ as $N \rightarrow \infty$, where $\nu_N(B,\bar{\omega})$ is the
probability kernel defined in \eqref{def:probkern}. Then $\nu$ is the only Gibbs measure for $V$ supported
on $\Omega^{*}$.
\end{lemma}

\noindent

The next lemma characterizes a set of path functions $\bar{\omega} \in \Omega$ for which the convergence
$\mu_N(B,\hat{\omega}) \rightarrow \mu(B)$ holds. A sufficient condition is given in terms of the kernel
$u(t,x,y)$ and the ground state $\varphi_0$.

\begin{lemma}
\label{lm:conv}
Let $(-\Delta)^{\alpha/2} + V$ be a fractional Schr\"odinger operator with Kato-decomposable potential
$V$ and ground state eigenfunction $\varphi_0$. Suppose that for some $\bar{\omega} \in \Omega$
\begin{align}
\label{eq:convcond1}
\frac{u(N-T,\bar{\omega}(-N),x)u(N-T,y,\bar{\omega}(N))}{u(2N,\bar{\omega}(-N),\bar{\omega}(N))}
&\stackrel {N \to \infty} \longrightarrow  e^{2 \lambda_0 T} \varphi_0(x)\varphi_0(y)
\end{align}
holds uniformly in $(x,y) \in \R^d \times \R^d$ for every $T>0$. Then for all $T>0$ and $B \in \cF_T$,
$\mu_N(B,\bar{\omega}) \rightarrow \mu(B)$ as $N \rightarrow \infty$, where $\mu$ is the
$P(\phi)_1$-measure for $V$.
\end{lemma}

\begin{proof}
By the Markov property of the process $\proo X$ and (5) of Lemma \ref{lm:kernel} we have for $N > T$,
$B \in \cF_T$ and $\bar{\omega} \in \Omega$
\begin{equation}
\label{eq:convme}
\begin{split}
\mu_N(B,\bar{\omega})& = \frac{1}{Z_N(\bar{\omega}(-N), \bar{\omega}(N))}
\int_{\R^d} dx \:  \int_{\R^d} dy\:  \left( \int_{\Omega} e^{-\int_{-N}^{-T} V(X_s(\omega))ds}
d\nu^{\bar{\omega}(-N),x}_{[-N,-T)}(\omega) \right.\\
& \qquad \times \int_{\Omega} \textbf{1}_B(\omega) e^{-\int_{-T}^{T} V(X_s(\omega))ds}d\nu^{x,y}_{[-T,T)}(\omega)
\left.\int_{\Omega} e^{-\int_{T}^{N} V(X_s(\omega))ds}d\nu^{y,\bar{\omega}(N)}_{[T,N)}(\omega)\right) \\
& = \int_{\R^d} dx \:  \int_{\R^d} dy\: \frac{u(N-T,\bar{\omega}(-N),x)u(N-T,y,\bar{\omega}(N))}
{u(2N,\bar{\omega}(-N),\bar{\omega}(N))} \\
& \qquad \times \int_{\Omega} \textbf{1}_B(\omega) e^{-\int_{-T}^{T} V(X_s(\omega))ds}d\nu^{x,y}_{[-T,T)}(\omega).
\end{split}
\end{equation}
Put $\Omega_M := \left\{\omega \in \Omega: \max(|\omega(-T)|, |\omega(T)|) < M \right\}$,
$M \in \N$. Clearly, $\Omega_M \nearrow \Omega$ when $M \rightarrow \infty$. If $B \subset \Omega_M$ for
some $M>1$, then the last factor in the above integral is a bounded function of $x$ and $y$ with compact
support and the assertion of the lemma follows from \eqref{eq:convcond1}.

Let now $B \in \cF_T$ be arbitrary. Fix $\varepsilon >0$ and choose $M$ large enough such that
$\mu(\Omega^c_M) < \varepsilon/4$. Since the claim is true for all $\cF_T$-measurable subsets of $\Omega_M$,
in particular for $B_M = B \cap \Omega_M$ and $\Omega_M$, we find $N_0$ such that for all $N > N_0$
$$
|\mu_N(B_M, \bar{\omega}) - \mu(B_M)| < \varepsilon/4 \quad \textnormal{and} \quad
|\mu_N(\Omega_M, \bar{\omega}) - \mu(\Omega_M)| < \varepsilon/4.
$$
This gives $\mu_N(\Omega^c_M, \bar{\omega}) < \varepsilon/2$ for $N > N_0$, and hence
\begin{align*}
|\mu_N(B,\bar{\omega}) - \mu(B)|
& =
|\mu_N(B_M,\bar{\omega}) + \mu_N(B \backslash \Omega_M,\bar{\omega})- \mu(B_M) - \mu(B \backslash \Omega_M)| \\
& \leq
|\mu_N(B_M,\bar{\omega}) - \mu(B_M)| + \mu(\Omega^c_M) + \mu_N(\Omega^c_M, \bar{\omega}) \leq  \varepsilon,
\end{align*}
completing the proof.
\end{proof}

Note that the condition
\begin{align}
\label{eq:convcond2}
\lim_{N \rightarrow \infty} \sup_{(x,y) \in \R^d \times \R^d} \left(\left|\frac{\widetilde{u}(N-T,\bar{\omega}(-N),x)\widetilde{u}(N-T,y,\bar{\omega}(N))}
{\widetilde{u}(2N,\bar{\omega}(-N),\bar{\omega}(N))} - 1\right| e^{2 \lambda_0 T} \varphi_0(x)\varphi_0(y)\right) = 0
\end{align}
is equivalent to \eqref{eq:convcond1}, which will be useful below.

We now discuss uniqueness for potentials $V(x) \rightarrow \infty$ as $|x| \rightarrow \infty$. Our first main
result is the following sufficient condition.

\begin{theorem}[\textbf{Uniqueness on full space}]
\label{th:iugibbs}
Let $\mu$ be the $P(\phi)_1$-measure for the Kato-decomposable potential $V$. If the semigroup $\{T_t: t \geq 0\}$
is AIUC, then $\mu$ is the unique Gibbs measure for $V$ supported on the full space $\Omega$.
\end{theorem}

\begin{proof}
Lemma \ref{lm:iucons} implies that condition \eqref{eq:convcond2} is satisfied for every $\omega \in \Omega$.
The assertion of the theorem follows by Lemmas \ref{lm:conv} and \ref{lm:uniq}.
\end{proof}

\begin{corollary}[\textbf{Uniqueness criterion}]
\label{cor:iugibbs}
By using Theorem \ref{th:suff} we immediately conclude that if there exist $R>0$ and $C_{V,R}>0$ such that for
all $|x|>R$
\begin{align}
\label{cond:IU}
\frac{V(x)}{\log|x|} \geq C_{V,R},
\end{align}
holds, then $\mu$ is the unique Gibbs measure for $V$ supported on $\Omega$.
\end{corollary}
\noindent
Since AIUC depends only on the behaviour of the potential at infinity (cf. Theorem \ref{th:suff}) local
singularities and perturbations on bounded sets have no effect on the uniqueness of the Gibbs measure for
this class of $V$. Recall that we denote by $\Lambda$ the spectral gap of the operator $H_\alpha$.


\begin{theorem}[\textbf{Uniqueness on full measure subspace}]
\label{th:subspace}
Let $V$ be a Kato-decomposable potential and assume $\varphi_0 \in L^2(\R^d) \cap L^1(\R^d)$. Then the
$P(\phi)_1$-measure $\mu$ is the unique Gibbs measure supported on the subspace
$$
\Omega^{*} := \left\{\omega \in \Omega: \lim_{|N| \rightarrow \infty}
\frac{e^{-\Lambda |N|}}{\varphi_0(\omega(N))} = 0\right\}.
$$
\end{theorem}

\begin{proof}
By Theorem \ref{th:subspaceom} $\mu(\Omega^{*}) = 1$. It suffices to show that it is the only Gibbs
measure with this property. Lemma \ref{lm:proj} implies that for every $0<t < N$, $N - t \geq 2$,
$$
\sup_{x,y \in \Rd} |e^{\lambda_0 (N-T)} u(N - t,x,y) - \varphi_0(x)\varphi_0(y)| \leq C_{V,t} e^{-\Lambda N}.
$$
Thus for all $\omega \in \Omega^{*}$ and every $x,y \in \Rd$ we clearly get
$$
|\widetilde{u}(N-T,\omega(-N),x)-1|\varphi_0(x) \leq C_{V,T} \frac{e^{-\Lambda N}}{\varphi_0(\omega(-N))}
\rightarrow 0,
$$
$$
|\widetilde{u}(N-T,y,\omega(N))-1|\varphi_0(y) \leq C_{V,T} \frac{e^{-\Lambda N}}{\varphi_0(\omega(N))}
\rightarrow 0,
$$
$$
|\widetilde{u}(2N,\omega(-N),\omega(N))-1| \leq C_{V,T} \frac{e^{-2 \Lambda N}}
{\varphi_0(\omega(-N))\varphi_0(\omega(N))} \rightarrow 0
$$
as $N \rightarrow \infty$, which implies \eqref{eq:convcond2}. It follows from Lemmas \ref{lm:conv} and
\ref{lm:uniq} that $\mu$ is the unique Gibbs measure supported on $\Omega^{*}$.
\end{proof}

We now illustrate the above results by some examples.

\begin{example}
\rm{
Let $H_{\alpha} = (-\Delta)^{\alpha/2} + V$ be a fractional Schr\"odinger operator with potential
$$
V(x) = C_0 |x|^{\delta} + \frac{C_1}{ |x-x_1|^{\beta_1}} - \frac{C_2}{|x - x_2|^{\beta_2}}
$$
where $C_0 >0$, $C_1, C_2 \geq 0$, $x_1, x_2 \in \R^d$ and $\delta > 0$, $\beta_1, \beta_2 \geq 0$.
It is straightforward to check that if $0< \beta_1, \beta_2 < \alpha < d$ or $0< \beta_1, \beta_2 < 1 = d \leq \alpha$, then $V$ is Kato-decomposable.
An immediate consequence of Theorem \ref{th:iugibbs} is that the $P(\phi)_1$-measure $\mu$ is the only
Gibbs measure corresponding to the process $\proo X$ and the potential $V$ supported on $\Omega$. Moreover,
by Theorem \ref{th:subspaceom} and Corollary \ref{cor:subspace1} we obtain that the measure $\mu$ is in fact
supported by the subset of $\Omega$ consisting of all path functions $\omega$ such that for every $\theta>0$
$$
|\omega(N)| = o\left(|N|^\frac{1+\theta}{\delta + d + \alpha}\right).
$$
}
\end{example}

\begin{example}[\textbf{Potential well}]
\rm{
Let $d=1$, $\alpha \in [1,2)$ and
$$
V(x)=
\left\{
\begin{array}{ll} - a, & x \in [-b,b] \\
0,                 & x \in [-b,b]^c ,
\end{array} \right.
$$
where $a, b > 0$. It is proved in \cite[Th. V.1]{bib:CMS} that the operator $H_{\alpha} =
(-\Delta)^{\alpha/2} + V$ has a spectral gap $\Lambda > 0$ and a ground state
$\varphi_0$ corresponding to the eigenvalue $\lambda_0 < 0$. By using Theorems
\ref{th:subspaceom} and \ref{th:decaying} we obtain that the corresponding $P(\phi)_1$-measure
$\mu$ is supported on a subset of paths given by the growth condition
$$
|\omega(N)| = o\left(|N|^{\frac{1+\theta}{1 + \alpha}}\right), \quad \forall{\theta>0}.
$$
Moreover, it follows from Theorem \ref{th:subspace} that $\mu$ is the unique Gibbs measure
supported on the subspace of paths such that
$$
|\omega(N)| = o\left(\exp\left(\frac{\Lambda}{1 + \alpha}|N|\right)\right).
$$
However, we do not know whether on the full space $\Omega$ there exist any other Gibbs measures.
}
\end{example}

\bigskip
\noindent
\textbf{Acknowledgments:} It is a pleasure to thank T. Kulczycki and K. Bogdan for discussions and
valuable comments. JL thanks the hospitality of Wroc{\l}aw University of Technology, and KK thanks
the hospitality of Loughborough University. We both thank IHES, Bures-sur-Yvette, for splendid
hospitality where part of the manuscript has been prepared.

\end{document}